\documentclass[oneside,leqno,11pt]{article}
\usepackage{amssymb,amsmath,latexsym, stmaryrd,amsthm}

\newcommand{\diag}{\operatorname{diag}}

\def\gg{\mathfrak{g}}

\def\gm{\mathfrak{m}}

\def\gp{\mathfrak{p}}

\def\gX{\mathfrak{X}}

\def\gX{\mathfrak{X}}

\def\ggg{> \hskip -5 pt >}

\def\rank{{\rm rank\,}}

\def\pos{{\rm pos}}
\def\neg{{\rm neg}}
\def\nul{{\rm nul}}

\def\Span{{\rm Span}\,}

\def\tr{^t\hskip -2pt}

\def\C{\mathbb{C}}

\def\F{\mathbb{F}}

\def\H{\mathbb{H}}

\def\N{\mathbb{N}}

\def\R{\mathbb{R}}

\def\Z{\mathbb{Z}}

\def\cB{\mathcal{B}}
\def\cC{\mathcal{C}}

\def\cF{\mathcal{F}}
\def\cG{\mathcal{G}}

\def\cI{\mathcal{I}}
\def\cJ{\mathcal{J}}

\def\cL{\mathcal{L}}
\def\cM{\mathcal{M}}

\def\cO{\mathcal{O}}
\def\cP{\mathcal{P}}

\def\cZ{\mathcal{Z}}

\newtheorem{theorem}[equation]{Theorem}

\newtheorem{lemma}[equation]{Lemma}

\newtheorem{corollary}[equation]{Corollary}

\newtheorem{proposition}[equation]{Proposition}

\newtheorem{definition}[equation]{Definition}

\newtheorem{example}[equation]{Example}

\newtheorem{remark}[equation]{Remark}

\def\sideremark#1{\ifvmode\leavevmode\fi\vadjust{\vbox to0pt{\vss
 \hbox to 0pt{\hskip\hsize\hskip1em
\vbox{\hsize2cm\tiny\raggedright\pretolerance10000 
 \noindent #1\hfill}\hss}\vbox to8pt{\vfil}\vss}}} 

\title{Cycle Spaces of Infinite Dimensional Flag Domains}

\author{Joseph A. Wolf\footnote{Research partially supported by a Dickson 
Emeriti Professorship and by a Simons Foundation grant.}}

\date{April, 2016}

\begin{document}
\maketitle

\abstract{Let $G$ be a complex simple direct limit group, specifically
$SL(\infty;\C)$, $SO(\infty;\C)$ or $Sp(\infty;\C)$.  Let 
$\cF$ be a (generalized) flag in $\C^\infty$.  If $G$ is  
$SO(\infty;\C)$ or $Sp(\infty;\C)$ we suppose further that $\cF$ is
isotropic.  Let $\cZ$ denote the corresponding flag manifold; thus 
$\cZ = G/Q$ where $Q$ is a parabolic subgroup of $G$.  In a recent paper
\cite{IPW2015} we studied real forms $G_0$ of $G$ and properties of their orbits
on $\cZ$.  Here we concentrate on open $G_0$--orbits $D \subset \cZ$.  When
$G_0$ is of hermitian type we work out the complete $G_0$--orbit structure of
flag manifolds dual to the bounded symmetric domain for $G_0$\,.  Then we
develop the structure of the corresponding cycle spaces $\cM_D$.  Finally we
study the real and quaternionic analogs of these theories.  All this extends 
results from the finite dimensional cases on the structure of hermitian 
symmetric spaces and cycle spaces (in chronological order:
\cite{W1969}, \cite{WW1977}, \cite{W1992}, \cite{W1994}, \cite{WZ2000},
\cite{W2000}, \cite{HS2001}, \cite{HW2002}, \cite{WZ2003}, \cite{HW2010}).}

\section{Introduction.}\label{sec1}
\setcounter{equation}{0}

The object of this paper is the study of certain infinite dimensional 
bounded symmetric domains and the related cycle spaces for open real group 
orbits on complex flag manifolds.  The cycle space theory is well understood 
in the finite dimensional setting (in chronological order: 
\cite{W1969}, \cite{WW1977}, \cite{W1992}, 
\cite{W1994}, \cite{WZ2000}, \cite{W2000}, \cite{HS2001}, \cite{HW2002}, 
\cite{WZ2003}, \cite{HW2010}).  Here we initiate its extension to infinite 
dimensions.  Specifically,
we look at the action of real reductive direct limit groups, 
$G_0$ such as $SL(\infty;\R)$, $SO(\infty,\infty)$, $Sp(\infty,q)$, 
or $Sp(\infty;\R)$, on a class of direct
limit complex flag manifolds $\cZ = G/Q$, where $G$ is the complexification 
of $G_0$\,.  While the classical finite dimensional 
setting \cite{W1969} is the guide, the results in infinite dimensions 
are much more delicate, and often different.  See \cite{IPW2015}, as
indicated below. 
In fact there are even stringent requirements for the existence of open 
$G_0$--orbits on $\cZ$.  In all cases where $G_0$ is the group of an 
hermitian symmetric space we work out a complete structure theory for the cycle
spaces of open orbits in our class of flag manifolds.  That structure is 
explicit in terms of the bounded symmetric domains of the $G_0$\,.
\smallskip

In Section \ref{sec2} we review the basic facts about our class of infinite
dimensional complex Lie groups, their construction, their flag manifolds, and
their real forms.   We note \cite{IPW2015} 
that every $G_0$--orbit on $\cZ$ is infinite dimensional, and we
describe just when the number of $G_0$--orbits on $\cZ$ is finite. 
\smallskip

In Section \ref{sec3} we concentrate on the cases where $G_0$ is a 
special linear group or is defined by a
bilinear or hermitian form.  We then recall foundational results from
\cite{IPW2015} and
describe a notion of nondegeneracy for flags $\cF \in \cZ$ (even in the 
cases $G_0 = SL(\infty;\R)$ and $G_0 = SL(\infty;\H$)). We use 
nondegeneracy to determine
which $G_0$--orbits are open, and in fact  and whether there are any open 
$G_0$--orbits.
\smallskip

In Section \ref{sec4} we develop a complete structure theory for the finitary
infinite dimensional bounded symmetric domains.  The results are similar to the
classical finite dimensional results, but one has to be careful about the
details.  We obtain complete extensions of the orbit structure (in 
particular the boundary structure) from the finite dimensional cases 
(\cite{KW1965}, \cite{WK1965}, \cite{W1969}).
\smallskip

Then in Section \ref{sec5} we initiate the study of cycle spaces of the 
open $G_0$--orbits on $\cZ$.  We start with the important case of 
$G_0 = SU(\infty,q)$, $q \leqq \infty$, using an idea from the finite 
dimensional setting.  We show how that idea leads to a precise description 
of the cycle space more generally.  This is the start of a program to extend
results of \cite{FHW2005} to infinite dimensions.
This study raises many important questions and initiates several
promising lines of research.  Compare \cite{FHW2005}.
\smallskip

One could carry out the considerations of Sections \ref{sec4} and \ref{sec5}
in a more unified way, but there are many small differences of technical 
detail, so it would not be advantageous.
\smallskip

Finally in Section \ref{sec6} we carry some of 
the results of Sections \ref{sec4} and
\ref{sec5} over to certain real and quaternionic bounded symmetric domains.
As noted in \cite{SW1975} this has some physical interest.
\smallskip

This study grew out of a joint project \cite{IPW2015}
with Ivan Penkov and Mikhail Ignatyev, where we studied real forms $G_0$ of 
$SL(\infty;\C)$ and the basic properties of their orbits on flag varieties 
$\cZ$.  I thank Ivan Penkov for important discussions on early versions of this
manuscript, and I thank the referee for the publication version of this paper
for useful critical comments.

\section{Basics.}\label{sec2}
\setcounter{equation}{0}

In this section we review some basic facts about our class of infinite
dimensional real and complex Lie groups, complex flag manifolds, and
real group orbits.

\subsection{Direct Limit Groups.}\label{ssec2a} 
\setcounter{equation}{0}
\smallskip

Let $V$ be a countable dimensional complex vector space and $E$ a fixed
basis of $V$.  We fix a linear order on $E$, specifically by $\N = \Z^+$,
where $E = \{e_1, e_2, \dots \}$.  When we come to flags and parabolics we
will consider other orders on $E$, but we use the given order by $\Z^+$
to define our groups and our exhaustions of $V$.
\smallskip

Let $V_*$ denote the span of the dual system
$\{e_1^*, e_2^*, \dots \}$; we view $V_*$ as the restricted
dual of $V$.  The group $GL(V,E)$ is the group of invertible linear
transformations on $V$ that keep fixed all but finitely many elements
of $E$.  It is easy to see that $GL(V,E)$ depends only on the pair $(V,V_*)$
as long as $V_*$ is constructed from $E$.
\smallskip

Express the basis $E$ as an increasing union $E = \bigcup E_n$ of finite
subsets.  That exhausts $V$ by finite dimensional subspaces 
$V_n = \Span\{E_n\}$, $V = \varinjlim V_n$\,, and thus expresses 
$GL(V,E)$ as $\varinjlim GL(V_n)$ and $SL(V,E)$ as $\varinjlim SL(V_n)$.
When we write $GL(\infty;\C)$ or $SL(\infty;\C)$ we must have in mind such 
an associated exhaustion of $V$ by finite dimensional subspaces.
\smallskip

For the orthogonal and symplectic groups,
$V$ is endowed with a nondegenerate symmetric or antisymmetric bilinear
form  $b$ that is related to $E$ as follows:  We can choose the
increasing union $E = \bigcup E_n$ so that the
$V_n = \Span\{E_n\}$ are nondegenerate for $b$, and so that
$b(e_m,V_n) = 0$ for $e_m \notin E_n$.  Thus 
$O(V,E,b) = \varinjlim O(V_n,b|_{V_n})$ 
when $b$ is symmetric, and
$Sp(V,E,b) = \varinjlim Sp(V_n,b|_{V_n})$ when $b$ is antisymmetric.
Again, when we write $O(\infty;\C)$, $SO(\infty;\C)$ or $Sp(\infty;\C)$
we must have in mind such an associated exhaustion of $V$ by finite 
dimensional $b$--nonsingular subspaces.

\subsection{Flags.}\label{ssec2b}
\setcounter{equation}{0}
\smallskip
We now recall some basic definitions from \cite{DP2004}.
A {\bf chain} of subspaces in $V$ is a
set $\cC$ of distinct subspaces such that if $F, F' \in \cC$ then either
$F \subset F'$ or $F' \subset F$.  We write $\cC'$ (resp. $\cC''$) for the
subchain of all $F \in \cC$ with an immediate successor (resp. immediate
predecessor).  Also, we write $\cC^\dagger$ for the set of all pairs $(F',F'')$
where $F'' \in \cC''$ is the immediate successor of $F' \in \cC'$.
\smallskip

Let $\cF$ be a chain, and let $\cF'$ and $\cF''$ be defined as just above.
Then $\cF$ is a {\bf generalized flag} if $\cF = \cF' \cup \cF''$ and 
$V\setminus \{0\} = \bigcup_{(F',F'') \in \cF^\dagger} (F'' \setminus F')$.
Note that $0 \ne v \in V$ determines 
$(F',F'') = (F'_v,F''_v) \in \cF^\dagger$ such that
$v \in F'' \setminus F'$.  If $\cF$ is a generalized flag then each of $\cF'$
and $\cF''$ determines $\cF$:
$$
\text{ if } (F',F'') \in \cF^\dagger \text{ then } 
        F' = {\bigcup}_{G'' \in \cF'', G'' \subsetneqq F''} G'' \text{ and }
        F'' = {\bigcap}_{G' \in \cF', G' \supsetneqq F'} G'.
$$

A generalized flag $\cF$ is {\bf maximal} if it is not properly contained in
another generalized flag.  This is equivalent to the condition that
$\dim F''_v/F'_v = 1$ for all $0 \ne v \in V$.
\smallskip

A generalized flag is 
a {\bf flag} if, as a linearly ordered set, the proper subspaces of 
$\cF$ are isomorphic to a linearly ordered subset of $\Z$, so that
we don't have to deal with limit ordinals.
\smallskip

In the orthogonal and symplectic cases,
we say that a generalized flag $\cF$ in $V$ is {\bf isotropic} (relative to
$b$) if $b(F,F) = 0$ for every $F \in \cF$.  This is equivalent to the 
notion in \cite{IPW2015}, where ``isotropic'' is defined to mean that 
$\tau: F \mapsto F^\perp$ (relative to $b$) is an order--reversing
involution of $\cF$, so that $(F',F'') \in \cF^\dagger$ if and only if
$((F'')^\perp,(F')^\perp) \in \cF^\dagger$.   In effect, if $\cF = (F_\alpha)$
is isotropic in the sense of this paper then $\cF \cup \cF^\perp :=
\cF \bigcup \{F^\perp \mid F \in \cF\}$ is isotropic in the sense of
\cite{IPW2015}, and if $\cJ = \{J_\beta\}$ is isotropic in the sense of
\cite{IPW2015} then $\{J_\alpha \mid J_\alpha \subset J_\alpha^\perp\}$ is
isotropic in the current sense.
\smallskip

A partial order  $\prec$ on a basis $E$ of $V$ is called {\bf strict} if
$\beta \prec \alpha$ implies $\beta \ne \alpha$, and
$\beta \preceq \alpha$ means that either $\beta \prec \alpha$ or
$\beta = \alpha$.  We emphasize that this is only a partial order, not a 
linear order, and there may be elements of the index set that are not 
comparable under $\prec$.  In particular $\prec$ need not be the same as
any order with which $E$ is presented.  See Example \ref{grass-ex} below.

\begin{definition}\label{compatible}{\em
A generalized flag $\cF$ is {\bf compatible} with $E$ if there exists a
strict partial order $\prec$ on $E$ for which every pair $(F',F'')$
is a pair $(\Span\{e_\beta \mid \beta \prec \alpha\}\,,\,
\Span\{e_\beta \mid \beta \preceq \alpha\})$ or a pair
$(0,\Span\{e_\beta \mid \beta \preceq \alpha\})$.
If $\cF$ is isotropic in the sense that each $F_\alpha$ is either isotropic
or coisotropic, then in addition we require that $E$ be isotropic.

A generalized flag $\cF$ is
{\bf weakly compatible} with $E$ if it is compatible with a basis $L$
of $V$ where $E\setminus (E\cap L)$ is finite.

A subspace $F \subset V$ is {\bf (weakly) compatible} with $E$ if the 
generalized flag $(0,F,V)$ is (weakly) compatible with $E$.

Generalized flags $\cF$ and $\cG$ are 
$E$--{\bf commensurable} if they are both weakly compatible
with $E$ and there is a bijection $\varphi: \cF \to \cG$ and a finite
dimensional $U \subset V$ such that each $F \subset \varphi(F)+U$,
$\varphi(F) \subset F + U$, and $\dim(F\cap U) = \dim(\varphi(F) \cap U)$.
$E$--commensurability is an equivalence relation.}
\hfill $\diamondsuit$
\end{definition}

\begin{example} \label{grass-ex}{\em 
This is the example that we'll need to discuss bounded symmetric domains.
Let $\cF = (0 \subset F \subset V)$.  We divide the index set $A$ of the
basis $E$ as
$A = A_1 \cup A_2$ where $A_1 = \{\alpha \mid e_\alpha \in F\}$.  Let
$\prec$ be any partial order on $A$ such that (i)
$\alpha_1 \prec \alpha_2$ whenever $\alpha_1 \in A_1$ and $\alpha_2 \in A_2$
and (ii) $A_i$ has a maximal element $\gamma_i$ in the sense that
$\alpha \prec \gamma_i$ whenever $\gamma_i \ne \alpha \in A_i$\,.  Then
$(0,F) = (0, \Span\{e_\beta \mid \beta \preceq \gamma_1\})$ (by convention on
pairs with $F' = 0$) and
$(F,V) = (\Span\{e_\beta \mid \beta \prec \gamma_2\}\,,
\Span\{e_\beta \mid \beta \preceq \gamma_2\})$, so $\cF$ is compatible
with $E$.

This example extends to generalized flags of 
the form $(0 \subset F_1 \subset \dots \subset F_\ell \subset V)$,
with only the obvious changes.}
\hfill $\diamondsuit$
\end{example}

Fix a generalized flag $\cF$ compatible with $E$.  
If $E$ is $b$--isotropic suppose that $\cF$ is isotropic.
Then $\cZ = \cZ_{\cF,E}$ denotes the {\bf flag manifold} $G/Q$ where $Q$ 
is the parabolic $\{g \in G \mid g(F) = F \text{ for all }F \in \cF\}$.  
If $E$ is isotropic we'll write $\cZ = \cZ_{\cF,b,E}$ for $G/Q$ where
$Q = Q_\cF$ is the stabilizer of $\cF$ in $G$.  As noted in Section 
\ref{ssec2c},
$\cZ_{\cF,E}$ is a holomorphic direct limit of finite dimensional complex
flag manifolds, so $\cZ_{\cF,E}$ has the structure of complex manifold.
\smallskip

Theorem 6.2 in \cite{DP2004} says

\begin{lemma}\label{wk-compat}
Let $\cF$ be a generalized flag that is weakly compatible with $E$.
If $G = SO(V,E,b)$ or $G = Sp(V,E,b)$ suppose further that $\cF$ is
isotropic.
If $g \in G$, then $g(\cF)$ is $E$--commensurable to $\cF$.
If $\cF$ and $\cL$ are $E$--commensurable then there is an element 
$g \in G$ such that $\cL = g(\cF)$.
\end{lemma}
\begin{proof} (Compare with Theorem 6.1 of \cite{DP2004}.)
If $g \in G$ then $V = U + W$ where $g$ is the identity on $W$, 
$g(U) = U$, and $\dim U < \infty$.  If $F_\alpha \in \cF$ then 
$g(F_\alpha) \subset F_\alpha + U$.  In particular $g(\cF)$ is
weakly compatible with $E$.  This proves the first statement.
\smallskip

Let $\cF$ and $\cL$ be $E$--commensurable, and let $U$ be a finite
dimensional subspace of $V$, such that each $F \subset \varphi(F)+U$,
$\varphi(F) \subset F + U$ and $\dim(F\cap U) = \dim(\varphi(F)+U)$.
They are weakly compatible with $E$ so they are compatible with bases
$X$ and $Y$ such that $E\setminus (E\cap X)$ and $E\setminus (E\cap Y)$
are finite.  Now $E \setminus (E\cap (X \cup Y))$ is finite; let $U$
denote its span and let $W$ be the span of its complement in $E$.  Let
$g \in G$ be the identity on $W$, and define $g: U \to U$ by 
$g(x_\alpha) = y_\alpha$ for
$\alpha$ an index of $E \setminus (E\cap (X \cup Y))$.
\end{proof}

\subsection{Flag Manifolds.}\label{ssec2c}
\setcounter{equation}{0}
\smallskip

Let $\cF$ be a generalized flag weakly compatible with $E$.  
If $G = SO(V,E,b)$ or $G = Sp(V,R,b)$, suppose that $\cF$ is isotropic.
In view of Lemma \ref{wk-compat},
\begin{remark}\label{def-flag-manifold}{\em
The {\bf flag manifold} $\cZ_{\cF,E}$ consists of all generalized flags 
in $V$ that are $E$--commensurable to $\cF$.} \hfill $\diamondsuit$
\end{remark}

Lemma \ref{wk-compat} says that $\cZ_{\cF,E}$ is a homogeneous space
for the complex group $G$.  Realize $V = \varinjlim V_n$ according to
an exhaustion $E = \bigcup E_n$ by finite subsets.  Denote
$\cF_n = \cF\cap V_n$\,.  In other words, if 
$\cF = \{F_\alpha\}_{\alpha \in A}\}$ then 
$\cF_n$ is $\{F_\alpha \cap V_n \}_{\alpha \in A}\}$ with repetitions 
allowed.  Now $\cF_n$ is a flag in $V_n$ so we have the flag manifold
$\cZ_{\cF_n,E_n}$\,.  Note that the
$F_\alpha \cap V_n \hookrightarrow F_\alpha \cap V_m$, $m \geqq n$,
define maps $\cZ_{\cF_n,E_n} \to \cZ_{\cF_m,E_m}$ and give us a direct
system $\{\cZ_{\cF_n,E_n}\}$ for which $\cZ_{\cF,E} = \varinjlim
\{\cZ_{\cF_n,E_n}\}$.  Since the finite dimensional flag manifold
$\cZ_{\cF_n,E_n}$ has the natural structure of homogeneous projective
variety under the action of $G_n$, and the 
$\cZ_{\cF_n,E_n} \to \cZ_{\cF_m,E_m}$ are equivariant rational maps and 
equivariant for $G_n \hookrightarrow G_m$, the infinite dimensional flag
manifold $\cZ_{\cF,E}$ is a $G$--homogeneous ind--variety.  We emphasize the
connection with the (finite dimensional) $\cZ_{\cF_n,E_n}$ by viewing
$\cZ_{\cF,E}$ as a complex ind--manifold referring to it
simply as a complex flag manifold.

\subsection{Real Forms of the Complex Groups.} \label{ssec2d}
\setcounter{equation}{0}
Corresponding to the complex classical groups $G$ mentioned above, we 
have their real forms as follows.  Here note that a local isomorphism 
to one of the groups on the following list implies an isomorphism of
Lie algebras, so the local isomorphism is compatible with 
the ind--structure specified as direct limit of finite dimensional Lie groups.
\smallskip

\noindent
\underline{If  $G = SL(\infty;\C)$},  then $G_0$ is locally isomorphic to 
one of
\smallskip

$SL(\infty;\R) = \lim_{n \to \infty} SL(n;\R)$ the real
special linear group,

$SL(\infty;\H) = \lim_{n \to \infty} SL(n;\H)$ the quaternion
special linear group,

$SU(p,\infty) = \lim_{n \to \infty} SU(p,n)$ the complex special
unitary group of finite real rank $p$, and

$SU(\infty,\infty) = \lim_{p,q \to \infty} SU(p,q)$ the complex
special unitary group of infinite real rank.
\pagebreak

\noindent
\underline{If  $G = GL(\infty;\C)$},  then $G_0$ is locally isomorphic to 
one of
\smallskip

$GL(\infty;\R) = \lim_{n \to \infty} GL(n;\R)$ the real
general linear group,

$GL(\infty;\H) = \lim_{n \to \infty} GL(n;\H) = SL(\infty;\H)
\times \R$ the quaternion general linear group,

$U(p,\infty) = \lim_{n \to \infty} U(p,n)$ the complex 
unitary group algebra of finite real rank $p$, and

$U(\infty,\infty) = \lim_{p,q \to \infty} U(p,q)$ the complex
unitary group of infinite real rank.
\smallskip

\noindent
\underline{If  $G = SO(\infty;\C)$},  then $G_0$ is locally isomorphic to 
one of
\smallskip

$SO(p,\infty) = \lim_{n \to \infty} SO(p,n)$ the real orthogonal
group of finite real rank $p$,

$SO(\infty,\infty) = \lim_{p,q \to \infty} SO(p,q)$ the real
orthogonal group of infinite real rank, and

{\it Caveat: when we write $SO($---$)$ we mean the topological identity 
component of $O($---$)$}.

$SO^*(\infty) = \lim_{n \to \infty} (SO^*(2n)
= \{g \in SL(n;\H) \mid g \text{ preserves }
\kappa(x,y) := \sum \bar x^\ell i y^\ell = {}^t\bar x i y\})$.
\smallskip

\noindent
\underline{If  $G = Sp(\infty;\C)$},  then $G_0$ is locally isomorphic to 
one of
\smallskip

$Sp(\infty;\R) = \lim_{n \to \infty} Sp(n;\R)$ the real
symplectic group,

$Sp(p,\infty) = \lim_{n \to \infty} Sp(p,n)$ the quaternion unitary
Lie algebra of finite real rank $p$, and

$Sp(\infty,\infty) = \lim_{p,q \to \infty} Sp(p,q)$ the quaternion
unitary Lie algebra of infinite real rank.
\medskip

As usual we use Roman letters for the Lie groups and the corresponding
lower case fraktur for their Lie algebras.
In order to be precise about the real groups we must be careful about two
notions: nondegeneracy of subspaces, and the role of $V$ and $E$ in
complex conjugation $\tau$ of $\gg$ over $\gg_0$ and $G$ over $G_0$\,.

\section{Basis and Exhaustion.}\label{sec3}
\setcounter{equation}{0}
We run through the real groups of Section \ref{ssec2d}, defining some
particular bases, flags and signatures
relevant to our results on cycle spaces.

\noindent
\subsection{$SU(\infty,q)\,, q \leqq \infty$.}\label{ssec3a}
\setcounter{equation}{0}
In this case $V = \C^{\infty,q}$  with $q \leqq \infty$ 
and we start with an ordered basis 
\begin{equation}\label{su-basis}
\begin{aligned}
E = &\{\dots, e_{-2}, e_{-1}, e_1, e_2, \dots , e_q\} \text{ if } q < \infty,\\
E = &\{\dots, e_{-2}, e_{-1}, e_1, e_2 , \dots \} \text{ if } q = \infty,
\end{aligned}
\end{equation}
where $G_0$ is defined by the hermitian form 
\begin{equation}\label{hform-su}
h(e_i,e_j) = \delta_{i,j} \text{ for } i < 0 \text{ and } 
	h(e_i,e_j) = -\delta_{i,j} \text{ for } i > 0.
\end{equation}
The corresponding exhaustion $V = \bigcup V_n$ realizes $G_0$
as ${\lim}_{k,\ell \to \infty} SU(k,\ell)$ or ${\lim}_{k \to \infty} SU(k,q)$.
\smallskip

$\cF$ is a flag in $V$ compatible with the ordered basis $E$ 
of (\ref{su-basis}).  {\em The partial order $\prec$ for this compatibility is
not necessarily the order of {\rm (\ref{su-basis})}; it is a
property of $\cF$ relative to $E$ rather than a property of the
ordering {\rm (\ref{su-basis})} of $E$.}
To each flag $\cF^{(1)} \in \cZ_{\cF,E}$ we assign the signature sequence
$\{s_k = s_k(\cF^{(1)}) 
	:= (\pos_k(\cF^{(1)}),\neg_k(\cF^{(1)}),\nul_k(\cF^{(1)}))\}$
where $\pos_k$ is the dimension of
the maximal positive definite subspace of $F^{(1)}_k$\,, $\neg_k$ is the
dimension of the maximal negative definite subspace, and $\nul_k$ is the
nullity.  If $\nul_k = 0$ we write $(\pos_k(\cF^{(1)}),\neg_k(\cF^{(1)}))$
for $(\pos_k(\cF^{(1)}),\neg_k(\cF^{(1)}),0)$. 
If the $F_k$ all are finite dimensional, then $\pos_k$, $\neg_k$ and 
$\nul_k$  all are finite.  If $q < \infty$, i.e. if $G_0$ has finite real 
rank $q$, then every $\nul_k \leqq q$.  However, when one or more of the
$F_k$ is infinite dimensional the signature sequence is not always useful.

\subsection{$SO(\infty,q)\,, q \leqq \infty$.}\label{ssec3b}
\setcounter{equation}{0}
Again $V = \C^{\infty,q}$. In terms of an ordered basis $E' = \{e_i'\}$
as in (\ref{su-basis}),
$G_0$ is defined by a symmetric bilinear form $b$ together with a 
hermitian form $h$, as follows:
\begin{equation}\label{so-forms}
\begin{aligned}
&b(e'_i,e'_j) = +\delta_{i,j} = h(e'_i,e'_j) \text{ for } i < 0, \,\,
b(e'_i,e'_j) = -\delta_{i,j} = h(e'_i,e'_j) \text{ for } i > 0, \\
&\text{the other } b(e'_k,e'_\ell) = 0 = h(e'_k,e'_\ell).
\end{aligned}
\end{equation}
To see that (\ref{so-forms}) defines $SO(\infty,q)$, we note that 
$SO(\infty,q)$ consists of all finitary real matrices (relative to the 
basis $E'$) in the $SO(\infty;\C)$ defined by $b$, and also consists of all 
real matrices in the $SU(\infty,q)$ defined by $h$.  Write $B$ and $H$ for the
matrices of $b$ and $h$, so $SO(\infty;\C)$ is given by $gB\cdot\tr g = B$
and $SU(\infty,q)$ is given by $gH\cdot\tr \bar g = H$.  Since $B = H$,
now $g \in SO(\infty;\C) \cap SU(\infty,q)$ implies $g = \bar g$ so
$g \in SO(\infty,q)$, and obviously $g \in SO(\infty,q)$ implies
$g \in SO(\infty;\C) \cap SU(\infty,q)$.  For $k, \ell \leqq \infty$
we have verified
\begin{equation}\label{so-def}
SO(k,\ell) = SO(k+\ell;\C) \cap SU(k,\ell)\,,  SO(k+\ell;\C) 
	\text{ defined by } b \,, SU(k,\ell) \text{ defined by } h.
\end{equation}

A $b$--isotropic flag in $V$ cannot be
compatible with $E'$ because a subspace of $V$ spanned by some of the $e'_i$
neither contains nor is contained in its $b$--orthocomplement.  So we define
\begin{equation}\label{so-basis}
\begin{aligned}
E = &\{\dots, e_{-2}, e_{-1}, e_1, e_2, \dots , e_q\} \text{ if } q < \infty,\\
E = &\{\dots, e_{-2}, e_{-1}, e_1, e_2 , \dots \} \text{ if } q = \infty,
\end{aligned}
\end{equation}
where 
\begin{equation}\label{forms-so}
\begin{aligned}
&h(e_i,e_j) = \delta_{i,j} \text{ for } i < 0,\, 
	h(e_i,e_j) = -\delta_{i,j} \text{ for } i > 0 \text{ and }\\
&b(e_i,e_j) = \delta_{i,j} \text{ for } i < -q,\,
	b(e_i,e_j) = \delta_{0,i+j} \text{ for }  i \geqq -q\,.
\end{aligned}
\end{equation}
The transformation $e'_i \mapsto e_i$ is not finitary when $q = \infty$,
but nonetheless every $g \in G$ is finitary relative to the basis $E$.
$\cF$ is an $E$--commensurable isotropic flag in $V$.
We use $h$ for the signature sequence 
$\{s_k = s_k(\cF^{(1)})
        := (\pos_k(\cF^{(1)}),\neg_k(\cF^{(1)}),\nul_k(\cF^{(1)}))\}$
for a flag $\cF^{(1)} \in \cZ_{\cF,E}$, where $\pos_k$ is the dimension of
the maximal $h$--positive definite subspace of $F^{(1)}_k$\,, $\neg_k$ is the
dimension of the maximal $h$--negative definite subspace, and $\nul_k$ is the
$h$--nullity.  As in Section \ref{ssec3a} above, 
if $\nul_k = 0$ we write $(\pos_k(\cF^{(1)}),\neg_k(\cF^{(1)}))$
for $(\pos_k(\cF^{(1)}),\neg_k(\cF^{(1)}),0)$, and
if the $F_k$ all are finite 
dimensional, then $\pos_k$, $\neg_k$ and
$\nul_k$  all are finite, and if $q < \infty$ then every $\nul_k \leqq q$.
\smallskip

\begin{remark}\label{orientation}{\rm
Orientation can be a consideration for $SO(\infty,q)$.
Following \cite[Theorem 2.8]{DPW2009}, the stabilizer of a $b$--isotropic 
flag $\cF$ determines 
all the subspaces in $\cF$ except when some there is an isotropic subspace
$L \in \cF$ with $\dim L^\perp/L = 2$.  In that case there are two maximal
isotropic subspaces $M_1$ and $M_2$ of $(V,b)$ that contain $L$, and 
there are three flags with the same stabilizer as $\cF$, and of course
$\cF$ is one of them.  We list them with {\em ad hoc} designations.  
(i) (undecided orientation) $\{F^{(1)} \in \cF \mid F^{(1)} \subset L$ or 
$L^\perp \subset F^{(1)}\}$ and neither of the $M_i$ is contained in $\cF$, 
(ii) (positive orientation) $\{F^{(1)} \in \cF \mid F^{(1)} \subset L$ or 
$L^\perp \subset F^{(1)}\} \cup \{M_1\}$, i.e. $M_1 \in \cF$, and
(iii) (negative orientation) $\{F^{(1)} \in \cF \mid F^{(1)} \subset L$ or 
$L^\perp \subset F^{(1)}\} \cup \{M_2\}$, i.e. $M_2 \in \cF$.  Signature does 
not distinguish these three flags, for example (ii) and (iii) have the 
same signatures for all $q$, and sometimes all three have the same signature
with $q = \infty$.}
\end{remark}

\subsection{$Sp(\infty,q)\,,\, q \leqq \infty$.}\label{ssec3c}
\setcounter{equation}{0}
Here $V = \C^{\infty,2q}$ and we use the basis (\ref{su-basis}) with $q$
replaced by $2q$.  Then $G_0$ is defined by both an antisymmetric 
bilinear form $b$ and an hermitian form $h$.
\begin{equation}\label{sp-forms}
\begin{aligned}
&b(e_{2i-1},e_{2i}) = -1,\, b(e_{2i},e_{2i-1}) = +1, 
	\text{ for } i > 0, \\
&b(e_{2i+1},e_{2i}) = +1,\, 
	b(e_{2i},e_{2i+1}) = -1 \text{ for } i < 0,
	\text{ all other } b(e_a,e_b) = 0; \\
&h(e_i,e_j) = \delta_{i,j} \text{ for } i < 0 \text{ and }
	h(e_i,e_j) = -\delta_{i,j} \text{ for } i > 0,
	\text{ all other } h(e_a,e_b) = 0.
\end{aligned}
\end{equation}
To see this we need the analog of (\ref{so-def}), and for that we need to
find the quaternion algebra that realizes $\C^{2p,2r}$ as $\H^{p,r}$.

\begin{lemma}\label{sp-def}
$Sp(p,r) = Sp(p+r;\C) \cap SU(2p,2r)$ where $SO(p+r;\C)$ is defined by $b$
as in {\rm (\ref{sp-forms})} and $SU(2p,2r)$ is defined by $h$ as in 
{\rm (\ref{sp-forms})}.
\end{lemma}

\begin{proof}  We work in matrices relative to the portion
$E = \{e_{-2p}, \dots, e_{2r}\}$ of (\ref{su-basis}).  Then $b$ has matrix
${B} = \diag\{J, \dots , J\}$ where $J = \left ( \begin{smallmatrix}
0 & -1 \\ 1 & 0 \end{smallmatrix} \right )$ and $h$ has matrix
${H} = \left ( \begin{smallmatrix} I_{2p} & 0 \\ 0 & -I_{2r} 
\end{smallmatrix} \right )$.  So $Sp(p+r;\C)$ is given by
$g{B}{\tr g} = {B}$ and $SU(2p,2r)$ is given by 
$g{H}{\tr\bar{g}} = {H}$.  Define $\R$--linear transformations
of $V$ by
$$
\cI(v) = \sqrt{-1}\, v \text{ and } \cJ(v) = \sqrt{-1} {B} {H} \bar v
\text{ for } v \in V.
$$
Compute
$$
\cI^2 = -I, \,\,\cJ^2 = -I \text{ and } \cI \cJ + \cJ \cI = 0
$$
so $\cI$ and $\cJ$ generate a quaternion algebra; call it $\H$.  If
$g \in Sp(p+r;\C) \cap SU(2p,2r)$, so $\tr g = {B}^{-1}g^{-1}{B}$ and 
$\tr \bar g = {H}^{-1}g^{-1}{H}$, then ${B}^{-1} = -{B}$,
${H}^{-1} = {H}$, and we compute
$$
\begin{aligned}
\cJ g \cJ^{-1} v &= (\sqrt{-1} {B} {H})(\bar g)
			(\overline{-\sqrt{-1} {H} {B}\bar v}) 
			= -{B} {H} \bar g {H} {B} v \\
	&= -{B} {H}\cdot {H} {\tr g^{-1}} {H}\cdot 
			{H} {B} v = {B} {\tr g^{-1}} {B} v
			={B}\cdot {B} g {B}^{-1}\cdot {B} v
			= gv
\end{aligned}
$$ 
for $v \in V$.  Thus $\cJ$ commutes with every 
$g \in Sp(p+r;\C) \cap SU(2p,2r)$, in other words every 
$g \in Sp(p+r;\C) \cap SU(2p,2r)$ is $\H$--linear. 
That shows $Sp(p+r;\C) \cap SU(2p,2r) \subset Sp(p,r)$.  
\smallskip

On the other hand, $\sigma: g \mapsto \cJ g \cJ^{-1}$ is an involutive
automorphism on the underlying real structure of $Sp(p+r;\C)$.  The latter
is simply connected, so its fixed point set is connected.  But $\sigma$
fixes every element of $Sp(p,r)$, which is maximal among the connected
subgroups of $Sp(p+r;\C)$.  So now $Sp(p,r) \subset Sp(p+r;\C) \cap SU(2p,2r)$.
That completes the proof.
\end{proof}

Now take the limit on $p$, or on $p$ and $r$, to see how
$G_0$ is defined by the two forms $b$ and $h$ of (\ref{sp-forms}).
Let $\cF$ be a $b$--isotropic flag in $V$ compatible with the basis $E$ of
(\ref{su-basis}).  For that, note that $\Span\{e_i \mid i \text{ even}\}$
and $\Span\{e_i \mid i \text{ odd}\}$ are $b$--isotropic subspaces.
As in the previous cases one can discuss signature for flags 
$\cF^{(1)} \in \cZ_{\cF,E}$\,.

\subsection{$SO^*(\infty)$.}\label{ssec3d}
\setcounter{equation}{0}
This case is similar to the case of $SO(\infty,\infty)$, except that we use a
different bilinear form $b$.  The basis is
\begin{equation}\label{so*-basis}
E = \{\dots e_{-3}, e_{-2}, e_{-1}, e_1, e_2, e_3, \dots\} =
	\bigcup E_n \text{ where } E_n = \{e_{-n}, \dots, e_n\}.
\end{equation}
$G_0$ is defined by the symmetric bilinear form $b$ and the hermitian form $h$:
\begin{equation}\label{so*-forms}
b(e_i,e_j) = \delta_{i+j,0},\,\,
h(e_i,e_j) = \delta_{i,j} \text{ for } i < 0 \text{ and }
h(e_i,e_j) = -\delta_{i,j} \text{ for } i > 0.
\end{equation}
Thus $V = \bigcup V_n$ where $V_n = \Span\{E_n\}$ and
$G_0 = SO^*(\infty) = \varinjlim SO^*(2n)$ where $SO^*(2n)$ is the subgroup
of $SL(2n;\C)$ defined by the forms $b$ and $h$ of (\ref{so*-forms}).  To
check this, note that that subgroup of $SL(2n;\C)$ itself has maximal compact
subgroup isomorphic to $U(n)$.
\smallskip

\subsection{$Sp(\infty;\mathbb{R})$.}\label{ssec3e}
\setcounter{equation}{0}
This case is similar to the case of $SO^*(\infty)$, except that the bilinear
form $b$ is antisymmetric.  We use the same basis (\ref{so*-basis}), 
with bilinear form $b$ and hermitian form $h$:
\begin{equation}\label{spr-forms}
\begin{aligned}
&b(e_i,e_j) = \delta_{i+j,0} \text{ for } i < 0 \text{ and }
b(e_i,e_j) = -\delta_{i+j,0} \text{ for } i > 0;\\
&h(e_i,e_j) = \delta_{i,j} \text{ for } i < 0 \text{ and }
h(e_i,e_j) = -\delta_{i,j} \text{ for } i > 0.
\end{aligned}
\end{equation}
Here $G = Sp(\infty;\C)$ is defined by $b$.  One can view $G_0$ as the real
(relative to $\Span_\R(E)$) elements of $G$, but for our purposes it is
better to view it as $G \cap SU(\infty,\infty)$ where $U(\infty,\infty)$
is defined by the hermitian form $h$.  For that, it suffices to check
that $Sp(n;\R) = Sp(n;\C) \cap U(n,n)$, and to check that it suffices
to note that $Sp(n;\C) \cap SU(n,n)$ contains a $U(n)$ in the form
$\left ( \begin{smallmatrix} A & 0 \\ 0 & \tr \bar A^{-1} \end{smallmatrix}
\right )$.
\smallskip

As in the previous cases one can discuss signature for flags 
$\cF^{(1)} \in \cZ_{\cF,E}$\,.

\subsection{$SL(\infty;\mathbb{R})$ and $SL(\infty;\mathbb{H})$.}\label{ssec3f}
\setcounter{equation}{0}
Fix a real form $V_0$ of $V$ and an ordered basis 
$E = \{e_1,e_2,e_3, \dots \}$ of $V_0$.  Then
$SL(\infty;\R)$ is defined by complex conjugation
$\tau: v \mapsto \overline{v}$ of $V$ over a real form $V_0$, while
$SL(\infty;\H)$ is defined by a conjugate linear map $\tau: v \mapsto
\left ( \begin{smallmatrix}  0 & I \\ -I & 0 \end{smallmatrix}\right )
\overline{v}$ on $V$.  In terms of $E$,
\begin{quote}
Case $\F = \R$: each $\tau e_a = e_a$, so $E$ is an $\R$--basis of a real
        form $V_0$ of $V$

Case $\F = \H$: $\tau e_{2a-1} = -e_{2a}$ and $\tau e_{2a} = e_{2a-1}$,
        so each $\{e_{2a-1}, ie_{2a-1}, e_{2a}, ie_{2a}\}$ is an
        $\R$--basis of an $\H$--subspace of $V$
\end{quote}
In the finite dimensional case the signature 
for a generalized flag $\cF^{(1)}$ is $\{s_{i,j} = s_{i,j}(\cF^{(1)})\}$
where $s_{i,j}(\cF^{(1)})$ is the dimension of the maximal complex subspace
of $F^{(1)}_i \cap \tau F^{(1)}_j$ (\cite{HS2001} and \cite{HW2002}).  In the
infinite dimensional cases we will have to be more precise 
\cite[\S 5]{IPW2015}.

\subsection{Nondegeneracy and Open Orbits.}\label{ssec3g}
\setcounter{equation}{0}
Fix a basis $E$ of $V$ as in Sections \ref{ssec3a} through \ref{ssec3f},
and a flag $\cF$ in $V$ that is compatible with $E$.  
Except in the cases of $SL(\infty;\mathbb{R})$ and $SL(\infty;\mathbb{H})$,
we use signatures of generalized flags to distinguish real group orbits
on $\cZ_{\cF,E}$\,, as follows. 

\begin{definition}\label{su-nondegen}{\rm
Let $G_0$ be defined by a nondegenerate bilinear form $b$ or an hermitian
form $h$ or both.  Then we say that a flag $\cF^{(1)} \in \cZ_{\cF,E}$ 
is {\bf nondegenerate} if (i) for $SU(\infty,q)$, $SO(\infty,q)$ or
$Sp(\infty,q)$ each $F^{(1)}_\alpha$ is $h$-- (or $b$--) nondegenerate; and
(ii) for $Sp(n;\R)$ or $SO^*(\infty)$ each $F^{(1)}_\alpha$ is
$h$--nondegenerate.}
\hfill $\diamondsuit$
\end{definition}

\begin{theorem}\label{open-orbit-su}
Let $G_0$ be $SU(\infty,q)$, $SO(\infty,q)$, $Sp(\infty,q)$, $SO^*(\infty)$
or $Sp(\infty;\R)$ and consider a flag $\cF^{(1)} \in \cZ_{\cF,E}$.  Then 
$G_0(\cF^{(1)})$ is open in $\cZ_{\cF,e}$ if and only if $\cF^{(1)}$ is
nondegenerate.
\end{theorem}

\begin{proof}
An orbit $G_0(\cF^{(1)})$ in $\cZ_{\cF,e}$ is open just when one stays inside
the orbit under any sufficiently small perturbation of a finite number of
the $F^{(1)}$ in $\cF^{(1)}$.  Using the direct limit topology on $\cZ_{\cF,e}$
and the finite dimensional analog (\cite{W1969}, \cite{FHW2005}), the 
assertion follows.  This is the same argument as that of the first part
of \cite[Proposition 5.1]{IPW2015}.
\end{proof}

\begin{corollary} \label{inf-many-open}
There are open $G_0$--orbits on $\cZ_{\cF,E}$ if and only if
$\cZ_{\cF,E}$ contains a nondegenerate flag.  In particular, if
$G_0$ is $SU(\infty,q)$, $SO(\infty,q)$ or $Sp(\infty,q)$ with
$q < \infty$ then there are open $G_0$--orbits on $\cZ_{\cF,E}$.
\end{corollary}

The matter is subtler for the special and general linear groups, where
we don't have $b$-- or $h$--nondegeneracy for subspaces of $V$, and where
if $\dim V = \infty$ then the $\dim (F^{(1)}_i\cap\tau F^{(1)}_j)$ do not suffice.
Instead we use \cite[Definition 5.1]{IPW2015} as follows.
\smallskip

\begin{definition}\label{sl-nondegen} {\em
Let $G_0$ be $SL(\infty;\F)$, $\F = \R$ or $\H$.  Then
$G_0(\cF^{(1)}) \subset \cZ_{\cF,E}$ is {\bf nondegenerate} if
$F^{(1)}_i\cap\tau F^{(1)}_j$ fails to properly contain 
$F^{(2)}_i\cap\tau F^{(2)}_j$\,,
whenever $\cF^{(2)} \in \cZ_{\cF,E}$ and 
$F^{(1)}_i, F^{(1)}_j \in \cF^{(1)}$.}
\phantom{XXXXX}\hfill$\diamondsuit$
\end{definition}

The first consequence of this definition is

\begin{theorem}\label{sl-extremes}{\rm (\cite[Proposition 5.3]{IPW2015})}
Let $G_0$ be $SL(\infty;\R)$ or $SL(\infty;\H)$, and consider a flag 
$\cF^{(1)} \in \cZ_{\cF,E}$.  Then the orbit $G_0(\cF^{(1)})$ is open
in $\cZ_{\cF,E}$ if and only if $\cF^{(1)}$ is nondegenerate.  In
particular, if each
$F_n^{(1)} \cap \tau F_n^{(1)} = 0$ then $G_0(\cF^{(1)})$ is open
in $\cZ_{\cF,E}$.
\end{theorem}

If $n$ is odd, or if $n = 2m$ and $\dim F \ne m$ for all 
$F \in \cF\cap \C^n$\,, then $G_{n,0} = SL(n;\R)$ has only one open 
orbit on a flag manifold $G_n/Q_n$ in $\C^n$;
if $n = 2m$ even, and some $F \in \cF\cap \C^n$ has dimension $m$, then 
there is an orientation question and $G_{n,0} = SL(n;\R)$ has two open orbits on
$G_n/Q_n$\,.  See \cite[Corollary 2.3]{HS2001}.  Further $G_{n,0} = SL(n;\H)$ 
has a unique open orbit on a flag manifold $G_n/Q_n$ in $\C^{2n}$.  
See \cite[Proposition 3.14]{HW2002}.
This extends to infinite dimensions as follows.

\begin{corollary}\label{one-open-sl}
Let $G_0$ be $SL(\infty;\R)$ or $SL(\infty;\H)$.  Then there is an
open $G_0$--orbit on $\cZ_{\cF,E}$ if and only if $\cZ_{\cF,E}$ contains
a nondegenerate flag, and in that case there is exactly one
open  $G_0$--orbit on $\cZ_{\cF,E}$.
\end{corollary}

\begin{proof} The first assertion is immediate from Theorem \ref{sl-extremes}.
For the second, let $\cO_1 = G_0(\cF^{(1)})$ and $\cO_2 = G_0(\cF^{(2)})$
be open $G_0$--orbits on $\cZ_{\cF,E}$.  Then 
$$
\begin{aligned}
\cZ_{\cF,E} = &\varinjlim \cZ_{\cF_n,E_n} \text{ where, for } n
		\text{ in a cofinal subset } \mathbb{S} \subset \Z^+, \\
	&E \text{ is an increasing union of finite subsets } E_n\,,\\
	&V_n = \Span\{E_n\}\text{ and }\cF_n = \cF\cap V_n := (F_k \cap V_n), \\
	&\cZ_{\cF_n,E_n} = G_n/Q_n \text{ flag manifold in } V_n \text{ with }
		Q_n \text{ parabolic in } G_n, \text{ and }\\
	& \cO_k \cap \cZ_{\cF_n,E_n} \text{ is an open } 
		G_{n,0}\text{-- orbit on } \cZ_{\cF_n,E_n}
		\text{ for } k = 1,\, 2.
\end{aligned}
$$
In the $SL(\infty;\R)$ case we modify $\mathbb{S}$.
If $n \in \mathbb{S}$ is even and $n+1 \notin \mathbb{S}$ we replace
$n$ by $n+1$.  If $n \in \mathbb{S}$ is even and $n+1 \in \mathbb{S}$ we
delete $n$.  Thus we may assume that every element of $\mathbb{S}$ is odd.
In the $SL(\infty;\H)$ case we do not modify $\mathbb{S}$.  Thus, in both
cases, if $n \in \mathbb{S}$ then $G_{n,0}$ has a unique open orbit on 
$\cZ_{\cF_n,E_n}$, so $(\cO_1 \cap \cZ_{\cF_n,E_n}) = 
(\cO_2 \cap \cZ_{\cF_n,E_n})$.  Thus $\cO_1$ meets $\cO_2$\,, so
$\cO_1 = \cO_2$\,.
\end{proof}

Combining the argument of the proof of Corollary \ref{one-open-sl} with the
uniqueness of closed orbits in the finite dimensional case \cite{W1969}, we
have the related result

\begin{proposition}\label{one-closed-sl}
{\rm (Compare \cite[Proposition 5.6]{IPW2015}.)}
Let $G_0$ be $SL(\infty;\R)$ or $SL(\infty;\H)$.  Then there is closed
$G_0$--orbit on $\cZ_{\cF,E}$ if and only if each $\tau F_i = F_i$\,,
and in that case there is exactly one closed $G_0$--orbit on $\cZ_{\cF,E}$.
\end{proposition}

\section{Complex Bounded Symmetric Domains}\label{sec4}
\setcounter{equation}{0}

The bounded symmetric domains are important cases of the orbits
considered in Section \ref{sec3}.  In finite dimensions they 
play a pivotal role in complex analysis, moduli theory, cycle 
space theory, automorphic function theory, and and both riemannian 
and complex differential geometry.  In this section we 
extend parts of the finite dimensional bounded domain theory 
to our infinite dimensional setting, following the lines of the classical 
examples in \cite{W1972}.  
\smallskip

In the classical theory one has the bounded symmetric domain $D_0 = G_0(z_0)$, 
its compact dual hermitian symmetric space $\cZ$, the Borel embedding
$D_0 \hookrightarrow \cZ$, and the Harish-Chandra embedding
$\xi^{-1}|_{D_0}: D_0 \hookrightarrow \gm^+$.  In the Harish-Chandra
embedding, $\gm^+ \subset \gg$ is a commutative subalgebra that represents the
holomorphic tangent space and $\xi: \gm^+ \to \cZ$ by $\xi(X) = \exp(X)z_0$\,.
A maximal set of strongly orthogonal noncompact positive roots 
$\{\alpha_1, \dots , \alpha_r\}$, $r = \rank D_0$\,, defines a set
$\{c_1, \dots , c_r\}$ of partial Cayley transforms, and the $G_0$--orbits
on $\cZ$ are exactly the $G_0(c_1\dots c_kc_{k+1}^2\dots c_{k+\ell}^2 z_0)$ 
where $k, \ell \geqq 0$ and $k+\ell \leqq r$.  The open orbits are the 
$G_0(c_1^2\dots c_\ell^2 z_0)$, i.e. the ones with $k = 0$, and
$G_0(c_1\dots c_r z_0)$ is the Bergman--Shilov boundary of $D_0$.  See
\cite{W1969}.  It is not so difficult to verify that this theory goes through 
{\em mutatis mutandis} for the infinite
dimensional bounded symmetric domains as well, with the one restriction that
$k + \ell < \infty$.
\smallskip

There are only four classes of (finitary) infinite dimensional complex
bounded symmetric domains: the $SU(\infty,q)/S(U(\infty)\times U(q))$
with $q \leqq \infty$, the $Sp(\infty;\R)/U(\infty)$, 
the $SO^*(\infty)/U(\infty)$, and the $SO(\infty,2)/[SO(\infty) \times SO(2)]$.
Their respective symmetric space ranks are $q$, $\infty$, $\infty$ and $2$.
In the all four cases it is easier to use some linear algebra, as in
the examples worked out in \cite{W1972}, than to stick to the general theory.
But of course we indicate the connection.
The fourth case, however, where $\cZ$ is a quadric in an infinite dimensional
complex projective space, is not as straightforward as the others.  Now we
run through the cases.

\subsection{The Complex Bounded Symmetric Domain for $SU(\infty,q)$.}
\label{ssec4a}
\setcounter{equation}{0}

We study the bounded symmetric domain $D_0$ associated to 
$G_0 = SU(\infty,q)$, $q \leqq \infty$.  Start with
$$
\begin{aligned}
&E = \{\dots, e_{-3}, e_{-2}, e_{-1}; e_1, e_2, \dots , e_q\} 
	\text{ for } q < \infty\\
&E = \{\dots, e_{-3}, e_{-2}, e_{-1}; e_1, e_2, e_3, \dots \}
	\text{ for } q = \infty
\end{aligned}
$$
where the hermitian form $h$ is given by
$$ 
h(e_i,e_j) = \delta_{i,j} \text{ for } i < 0,\,\,
h(e_i,e_j) = -\delta_{i,j} \text{ for } i > 0.
$$
Let $F = \Span\{e_i \mid i > 0\}$.  As in Example \ref{grass-ex}, 
$\cF = (0,F,V)$ is compatible with $E$.  Also, $G_0(\cF)$ is 
open in $\cZ_{\cF,E}$.  The bounded symmetric domain is 
$D_0 = G_0(\cF) \in \cZ_{\cF,E}$.
Note that $\cZ_{\cF,E}$ is a complex Grassmann manifold and the domain is
\begin{equation}\label{bdedsym-su}
D_0 = \{\cF^{(1)} = (0,F^{(1)},V) \in \cZ_{\cF,E} \mid F^{(1)} 
	\text{ is a maximal negative definite subspace of } V\}.
\end{equation}
We go on to see why it is a bounded symmetric domain.
\smallskip

We will use the $h$--orthogonal
decomposition $V = V_+ \oplus V_-$ where $V_+ = \Span\{e_i \mid i < 0\}$
and $V_- = \Span\{e_i \mid i > 0\}$  and the orthogonal projections
$\pi_\pm: V \to V_\pm$\,.  The kernel of $\pi_-$ is $h$--positive definite
so it has zero intersection with $F^{(1)}$ for any $\cF^{(1)} = (0,F^{(1)},V) \in D_0$.
Thus $\pi_-: F^{(1)} \cong V_-$ is injective.  Since $F^{(1)}$ is a maximal negative
definite subspace $\pi_-: F^{(1)} \cong V_-$ is surjective as well.  Now we
have a well defined linear map
\begin{equation}\label{def-so-z}
Z_{F^{(1)}}: V_- \to V_+ \text{ defined by } \pi_-(x) \mapsto \pi_+(x)
	\text{ for } x \in F^{(1)}.
\end{equation}
Since $\cF^{(1)}$ is weakly compatible with $E$, the matrix of $F^{(1)}$ relative
to $E$ has only finitely many nonzero entries.  In other words $Z_{F^{(1)}}$
is finitary.  Using $\pi_-: F^{(1)} \cong V_-$ and the basis $\{e_i \mid i > 0\}$
of $F = V_-$ we have a basis $\{e''_i\}$ of $F^{(1)}$ defined by $\pi_-(e''_i)
= e_i$\,.  Write $e''_i = e_i + \sum_{j<0} z_{j,i}e_j$; then
$(z_{j,i})$ is the matrix of $Z_{F^{(1)}}$.  The fact that $F^{(1)}$ is $h$--negative
definite, in other words $(h(e''_i,e''_\ell)) \ll 0$, translates to the
matrix condition $I - (z_{j,i})^*(z_{j,i}) \ggg 0$, equivalently the
operator condition $I - Z_{F^{(1)}}^* Z_{F^{(1)}} \ggg 0$.  
\smallskip

Conversely if $Z: V_- \to V_+$ is finitary and satisfies $I - Z^*Z \ggg 0$,
then the column span of its matrix relative to $E$ is a maximal negative 
definite subspace $F^{(1)}$, and $\cF^{(1)} = (0,F^{(1)},V) \in D_0$\,.
\smallskip

The block form matrices of elements of $G_0$ act by
$\left ( \begin{smallmatrix} A & B \\ C & D \end{smallmatrix}\right ) :
\left ( \begin{smallmatrix} Z \\ I \end{smallmatrix}\right ) \to
\left ( \begin{smallmatrix} AZ+B \\ CZ+D \end{smallmatrix}\right )$,
which has the same column span as
$\left ( \begin{smallmatrix} (AZ+B)(CZ+D)^{-1} \\ I \end{smallmatrix}\right )$.
So $G_0$ acts by linear fractional transformations,
$\left ( \begin{smallmatrix} A & B \\ C & D \end{smallmatrix}\right ) : Z
\to (AZ+B)(CZ+D)^{-1}$.  Now we summarize.

\begin{proposition}\label{bded-realization-su}
$D_0$ is realized as the bounded domain consisting of all finitary
$Z: V_- \to V_+$
such that $I - Z^*Z \ggg 0$.  In that realization the action of $G_0$ is 
$\left ( \begin{smallmatrix} A & B \\ C & D \end{smallmatrix}\right ) : Z
\to (AZ+B)(CZ+D)^{-1}$.
\end{proposition}
\medskip

\centerline{\bf Orbits}
\medskip

\noindent
There are $q+1$ open $G_0$--orbits on $\cZ_{\cF,E}$:
$$
\begin{aligned}
D_k = G_0(0,F_{(k)},V) \text{ where } &F_{(k)} = 
	{\Span}\{e_{-k}, \dots , e_{-1}; e_{k+1}, \dots , e_q\}
		\text{ if } q < \infty, \\
&F_{(k)} = {\Span}\{e_{-k}, \dots , e_{-1}; e_{k+1}, e_{k+2}, \dots\}
		\text{ if } q = \infty, \\
\end{aligned}
$$
If $q = 1$ then $D_1$ and $D_0$ are the upper and lower ``hemispheres'' in
an infinite version of the Riemann sphere; they are related by the square of
a Cayley transform.  If $q > 1$ then $D_0$ is the only convex $D_k$, but the
others are reached by squares of partial Cayley transforms  applied
to $F = F_0$ as in \cite{W1972}, \cite{KW1965} and \cite{WK1965}. 
\smallskip

In this bounded symmetric domain setting, the $G_0$--orbits on $\cZ_{\cF,E}$
of signature $(a,b,c) = (\pos,\neg,\nul)$ have $a$ and $c$ finite and
$\leqq q$ because each $\cF^{(1)} = (0,F^{(1)},V) \in \cZ_{\cF,E}$ is weakly compatible 
with $E$.  We denote those orbits by
\begin{equation}
\begin{aligned}
D_{a,b,c} &= G_0(0,(F_+ + F_- + F_0),V) \text{ where }\\
	&F_0 = \Span\{e_{-c}+e_c, \dots ,e_{-1}+e_1\} \text{ (null) }\\
	&F_+ = \Span\{e_{-c-a}, \dots , e_{-c-1}\} \text{ (positive)}\\
	&F_- = \Span\{e_{c+1}, \dots , e_{c+b}\} \text{ if } q < \infty,
		\Span\{e_{c+1},e_{c+2}, \dots\} \text{ if } q = \infty
			\text{ (negative)}.
\end{aligned}
\end{equation}
The open orbits are the $D_a = D_{a,b,0}\,, a < \infty$ and $a+b = q$.  
In other words, they are the ones for $c = 0$.  If $q < \infty$ there is
a unique closed orbit, $D_{0,0,q}$, consisting of the 
$\cF^{(1)} = \{F^{(1)}\} \in \cZ_{\cF,E}$ for which $F^{(1)}$ is null.  It is in the 
closure of every orbit. If $q = \infty$ there is no closed orbit.
\smallskip

One goes from the initial orbit $D_{0,q,0} = G_0(\cF)$ to any $D_{a,b,c}$ 
by applying a product of partial Cayley transforms to $F$.  Specifically,
$D_{a,b,c} = G_0(c_1\dots c_c c^2_{c+1}\dots c^2_{a+c} \cF)$ where the
partial Cayley transforms $c_k$ (corresponding to $0 \to 1 \to \infty
\to -1 \to 0$ in one variable) are given by
\begin{equation}\label{cayley-su}
c_k(e_{-k}) = \tfrac{1}{\sqrt{2}} (e_{-k} - e_k), \,\,
c_k(e_k) = \tfrac{1}{\sqrt{2}}(e_{-k} + e_k), \,\,
c_k(e_j) = e_j \text{ for } j \ne \pm k.
\end{equation}
Here $1 \leqq k \leqq q$ when $q < \infty$ and $1 \leqq k < \infty$ when
$q = \infty$.  In particular one reaches the boundary (of $D_0 = D_{0,q,0}$)
orbits by a product without repetition of $\leqq q$ partial Cayley transforms, 
and if $q < \infty$ the closed orbit is $D_{0,0,q} = G_0(c_1\dots c_q \cF)$.
If $q < \infty$ the closed orbit is the Bergman-Shilov boundary of $D_0$\,.
\smallskip

\subsection{The Complex Bounded Symmetric Domain for $Sp(\infty;\mathbb{R})$.}
\label{ssec4b}
\setcounter{equation}{0}

Now let $G_0 = Sp(\infty;\R)$. It is defined relative to the basis
$
E = \{\dots, e_{-3}, e_{-2}, e_{-1}; e_1, e_2, e_3, \dots \}
$
by the hermitian form $h$ and the antisymmetric bilinear form $b$,
$$
\begin{aligned}
&h(e_i,e_j) = \delta_{i,j} \text{ for } i < 0,\,\,
h(e_i,e_j) = -\delta_{i,j} \text{ for } i > 0 \text{ and } \\
&b(e_i,e_j) = \delta_{i+j,0} \text{ for } i < 0,\,\,
b(e_i,e_j) = -\delta_{i+j,0} \text{ for } i > 0.
\end{aligned}
$$
The domain $D_0$ consists of the maximal
$h$--negative definite $b$--isotropic subspaces of $V$ in $\cZ_{\cF,E}$\,. 
In other words, let
$F = \Span\{e_i \mid i > 0\}$.  Evidently $\cF = (0,F,V)$ is
compatible with $E$ and $G_0(\cF)$ is open
in $\cZ_{\cF,E}$.  The bounded symmetric domain is 
$$
D_0 := G_0(\cF) \subset \cZ_{\cF,E}.
$$
Note that $\cZ_{\cF,E}$ is contained in the complex Grassmann manifold 
of Section \ref{ssec4a} for $q = \infty$.
\smallskip

In Lie group terms, $D_0 \cong Sp(\infty;\R)/U(\infty)$ where $U(\infty)$ is
the stabilizer of $F$.  Let both $\cF$ and
$\cF_{(0)}$  denote the flag $(0,F,V)$, so $D_0 = G_0(\cF_{(0)})$\,.
The open $G_0$--orbits on $\cZ_{\cF,E}$ are the $D_k = G_0(\cF_{(k)})$ where
$$
F_{(k)} = \Span\{e_{-k}, e_{-k+1}, \dots , e_{-1};
e_{k+1}, e_{k+2}, \dots \} \text{ and } \cF_{(k)} = (0,F_{(k)},V)
$$ 
for integers $k \geqq 0$.  Note that
$D_k \cong Sp(\infty;\R)/U(k,\infty - k)$ where the $\infty - k$ refers
to the action on $\Span\{e_{k+1}, e_{k+2}, \dots \}$. 
Also, $F_{(k)}^\perp = F_{(k)}$ relative to $b$, so $\cF_{(k)}^\perp
= \cF_{(k)}$.
\smallskip

The corresponding $D_\infty$ is the $h$--orthocomplement of $D_0$\,, orbit of
$(0,F_{(\infty)},V)$ where $F_{(\infty)} :={\Span}\{e_i \mid i < 0\}$.
\smallskip

As in the $SU$ setting, the partial Cayley transforms $c_j$ are given by 
(\ref{cayley-su}) and one passes from $D_0$ to $D_k$ by
$F_{(k)} = c_1^2c_2^2\dots c_k^2 F_{(0)}$.  Compare \cite{W1972}.
Similarly, as in \cite{WK1965}, the boundary of $D_0$ is the union of the 
orbits $G_0(c_1c_2\dots c_\ell  F_{(0)})$, but here there is no closed
$G_0$--orbit on $\cF_{(k)}$, $k < \infty$, and thus no Bergman--Shilov 
boundary of $D_0$\,.
\smallskip

The calculations for $D_0$ to be a bounded symmetric domain are essentially
the same as those in Section \ref{ssec4a}.  The result is

\begin{proposition}\label{bded-realization-sp}
$D_0$ is realized as the bounded domain consisting of all finitary
$Z: V_- \to V_+$ such that the matrix of $Z$ is symmetric and 
$I - Z^*Z \ggg 0$.  In that realization the action of $G_0$ is
$\left ( \begin{smallmatrix} A & B \\ C & D \end{smallmatrix}\right ) : Z
\to (AZ+B)(CZ+D)^{-1}$.
\end{proposition}

\begin{corollary}\label{su2sp}
The bounded symmetric domain for $Sp(\infty;\R)$ is a totally geodesic
submanifold of the bounded symmetric domain for $SU(\infty,\infty)$.
\end{corollary}

As in Section \ref{ssec4a} for $SU(\infty,\infty)$, the $G_0$--orbit 
of signature $(pos,neg,nul)=(a,b,c)$, $a$ and $c$ finite, is
$$
\begin{aligned}
D_{a,b,c} &= G_0(0,(F_+ + F_- + F_0),V) \text{ where }\\
        &F_0 = \Span\{(e_{-c}+e_c), \dots ,(e_{-1}+e_1)\} \text{ ($h$--null),}\\
        &F_+ = \Span\{e_{-c-a}, \dots , e_{-c-1}\} \text{ ($h$--positive definite),}\\
        &F_- = \Span\{e_{c+1},e_{c+2}, \dots\} \text{ ($h$--negative definite)},
\end{aligned}
$$
and every $G_0$--orbit on $\cZ_{\cF,E}$ is one of those $D_{a,b,c}$.  We 
always have $b = \dim F_- = \infty$.  The open orbits are the case $c = 0$ 
mentioned above: $D_k = D_{k,b,0}$\,.  
\smallskip

\subsection{The Complex Bounded Symmetric Domain for $SO^*(\infty)$.}
\label{ssec4c}
\setcounter{equation}{0}
\smallskip

Next, we let $G_0 = SO^*(\infty)$. It is defined relative to the basis
$
E = \{\dots, e_{-3}, e_{-2}, e_{-1}; e_1, e_2, e_3, \dots \}
$
by the hermitian form $h$ and the symmetric bilinear form $b$,
$$
h(e_i,e_j) = \delta_{i,j} \text{ for } i < 0,\,\,
h(e_i,e_j) = -\delta_{i,j} \text{ for } i > 0, \text{ and }
b(e_i,e_j) = \delta_{i+j,0} \text{ for all } i, j.
$$
The domain $D_0$ consists of the maximal
$h$--negative definite $b$--isotropic subspaces of $V$ that are 
weakly compatible with $E$.  In other words, let
$F = \Span\{e_i \mid i > 0\}$.  Evidently $\cF = (0,F,V)$ is
compatible with $E$ and $G_0(\cF)$ is open
in $\cZ_{\cF,E}$.  The bounded symmetric domain is 
$$
D_0 := G_0(\cF) \subset \cZ_{\cF,E}.
$$
Again, $\cZ_{\cF,E}$ is contained in the complex Grassmann manifold 
of Section \ref{ssec4a} for $q = \infty$.
\smallskip

In Lie group terms, $D_0 \cong SO^*(\infty)/U(\infty)$ where $U(\infty)$ is
the stabilizer of $F$.  Let both $\cF$ and
$\cF_{(0)}$  denote the flag $(0,F,V)$, so $D_0 = G_0(\cF_{(0)})$\,.
The open $G_0$--orbits on $\cZ_{\cF,E}$ are the $D_k = G_0(\cF_{(k)})$ where
$$
F_{(k)} = \Span\{e_{-k}, e_{-k+1}, \dots , e_{-1};
e_{k+1}, e_{k+2}, \dots \} \text{ and } \cF_{(k)} = (0,F_{(k)},V)
$$ 
for integers $k \geqq 0$.  Note that
$D_k \cong SO^*(\infty)/U(k,\infty - k)$ where the $\infty - k$ refers
to the action on $\Span\{e_{k+1}, e_{k+2}, \dots \}$. 
Also, $F_{(k)}^\perp = F_{(k)}$ relative to $b$, so $\cF_{(k)}^\perp
= \cF_{(k)}$.
\smallskip

The corresponding $D_\infty := G_0(\cF_{(\infty)})$ where $F_{(\infty)}$ 
is the $h$--orthocomplement $\Span\{e_i \mid i < 0\}$ of $F_0$\,.
\smallskip

Here the partial Cayley transforms are not given by (\ref{cayley-su}), but
rather by
\begin{equation}\label{cayley-so*}
\begin{aligned}
&c_k(e_{-2k}) \,\,= \tfrac{1}{\sqrt{2}}(e_{-2k} - e_{2k}),\,
\phantom{XXXX}c_k(e_{-2k+1}) = \tfrac{1}{\sqrt{2}}(e_{-2k+1} + e_{2k-1}),\\
&c_k(e_{2k-1}) = \tfrac{1}{\sqrt{2}}(-e_{-2k+1} + e_{2k-1}),\,
\phantom{X}c_k(e_{2k}) \,\,\,\,\,\,\,\,\, = 
	\tfrac{1}{\sqrt{2}}(e_{-2k} + e_{2k}),\\
&c_k(e_j) = e_j \text{ for } j \notin \{-2k,-2k+1,2k-1,2k\}.
\end{aligned}
\end{equation}

As in the $SU$ and $Sp$ settings, one passes from $D_0$ to $D_k$ using
$F_{(k)} = c_1^2c_2^2\dots c_k^2 F_{(0)}$ where the $c_j$ are partial
Cayley transforms defined by (\ref{cayley-so*}), as in \cite{W1972}.  
Similarly, as in \cite{WK1965}, the boundary of $D_0$ is the union of the 
orbits $G_0(c_1c_2\dots c_\ell  F_{(0)})$, but here there is no closed
$G_0$--orbit on $\cF_{(k)}$, $k < \infty$, and thus no Bergman--Shilov 
boundary of $D_0$\,.
\smallskip

The calculations for $D_0$ to be a bounded symmetric domain are essentially
the same as those in Section \ref{ssec4a}.  The result is

\begin{proposition}\label{bded-realization-so*}
$D_0$ is realized as the bounded domain consisting of all finitary
$Z: V_- \to V_+$ such that the matrix of $Z$ is antisymmetric and 
$I - Z^*Z \ggg 0$.  In that realization the action of $G_0$ is
$\left ( \begin{smallmatrix} A & B \\ C & D \end{smallmatrix}\right ) : Z
\to (AZ+B)(CZ+D)^{-1}$.
\end{proposition}

\begin{corollary}\label{su2so*}
The bounded symmetric domain for $SO^*(\infty)$ is a totally geodesic
submanifold of the bounded symmetric domain for $SU(\infty,\infty)$.
\end{corollary}

As in Section \ref{ssec4b} for $Sp(\infty;\R)$, the $G_0$--orbit 
of signature $(pos,neg,nul)=(a,b,c)$, $a$ and $c$ finite, is
$$
\begin{aligned}
D_{a,b,c} &= G_0(0,(F_+ + F_- + F_0),V) \text{ where }\\
        &F_0 = \Span\{(e_{-c}+e_c), \dots ,(e_{-1}+e_1)\} \text{ ($h$--null),}\\
        &F_+ = \Span\{e_{-c-a}, \dots , e_{-c-1}\} \text{ ($h$--positive definite),}\\
        &F_- = \Span\{e_{c+1},e_{c+2}, \dots\} \text{ ($h$--negative definite)},
\end{aligned}
$$
and every $G_0$--orbit on $\cZ_{\cF,E}$ is one of those $D_{a,b,c}$.  We 
always have $b = \dim F_- = \infty$.  The open orbits are the case $c = 0$ 
mentioned above: $D_k = D_{k,b,0}$\,.  

\subsection{The Complex Bounded Symmetric Domain for $SO(\infty,2)$.}
\label{ssec4d}
\setcounter{equation}{0}
This one is more delicate because the bounded domain for $SO(\infty,2)$
does not sit as an easily described totally geodesic submanifold of the
bounded domain for any of the $SU(\infty,q)$.  Specifically, it is a 
bounded domain in a nondegenerate complex quadric in an infinite dimensional 
complex projective space.
\smallskip

We use a basis 
\begin{equation}\label{so2-basis}
E =\{\dots , e_{-3}, e_{-2}, e_{-1}, e_1, e_2\}
\end{equation}  
of $V$.  $G_0 = SO(\infty,2) = SO(\infty;\C)\cap U(\infty,2)$ is the 
connected real semisimple Lie group defined by
the following hermitian form $h$ and the symmetric bilinear form $b$:
\begin{equation}\label{so2-forms}
h(e_i,e_j) = \delta_{i,j} \text{ for } i < 0,\, 
	h(e_i,e_j) = -\delta_{i,j} \text{ for } i > 0,\, 
		b(e_i,e_j) = \delta_{i,j} \text{ for all } i,j.
\end{equation}
This is a finitary change from (\ref{so-basis}) and (\ref{forms-so}).
The only effect of the change 
is to facilitate our study of bounded domains and cycle spaces
for $SO(\infty,2)$.
\smallskip

For $n > 0$ we have $E_n = \{e_{-n}, e_{-n+1}, \dots , e_{-1}, e_1, e_2\}$,
the $(n+2)$--dimensional subspace $V_n = \Span(E_n)$ of $V$, the 
$(n+1)$--dimensional complex projective space $\cP^{n+1} = \cP(V_n)$, and the 
nondegenerate complex quadric $\cZ_n = \{[v] \in \cP^{n+1} \mid b(v,v) = 0\}$.
They define the infinite dimensional complex projective space
$\cP^\infty = \cP(V) = \varinjlim \cP^{n+1}$ and the nondegenerate complex
quadric $\cZ = \{[v] \in \cP^\infty \mid b(v,v) = 0\} = \varinjlim \cZ_n$
in $\cP^\infty$.  Note that everything here is finitary.  The complex
group $G = SO(\infty;\C)$ is transitive on $\cZ$ because $SO(n+2;\C)$ is
transitive on $\cZ_n$\,.
\smallskip

Our bounded symmetric domain will be $D_0 = G_0(z_0) \subset \cZ$ where
$z_0 = [e_1 + \sqrt{-1}\,e_2]$.  We now look at the Harish-Chandra
embedding of $D_0$ in its holomorphic tangent space.  The Lie algebra
$$
\gg = \left \{ \left . \left ( \begin{smallmatrix} A & B \\ -\tr B & D 
	\end{smallmatrix} \right ) \right |\, \tr A = -A,\, \tr D = -D
	\right \} \text{ where } A \in \C^{\infty \times \infty},\,
	B \in \C^{\infty \times 2},\, D \in \C^{2 \times 2} 
$$
and the isotropy subalgebra at $z_0$ is the parabolic
$$
\gp = \left \{ \left . \left ( \begin{smallmatrix} A & B \\ -\tr B & D
        \end{smallmatrix} \right ) \in \gg  \,\, \right |\, B = (B'',\sqrt{-1}\,B'')
	 \right \} \text{ where } B'' \in \C^{\infty \times 1} .
$$
The holomorphic tangent space to $\cZ$ at $z_0$ is
$$
\gm^+ = \left \{ \left . \left ( \begin{smallmatrix} 0 & B \\ -\tr B & 0
	\end{smallmatrix} \right ) \right |\, B = (\sqrt{-1}\,B'',B'')
	\text{ with } B'' \in \C^{\infty \times 1} \right \}.
$$
We view $\gm^+$ as $\C^\infty$ (column vectors) $Z$ under the correspondence
$$
Z \mapsto \widetilde{Z} := \left ( \begin{smallmatrix} 0 & Z' \\ -\tr Z' & 0
        \end{smallmatrix} \right ) \text{ where } Z' = (\sqrt{-1}\,Z,Z)
	\in \C^{\infty \times 2}.
$$
Computing as in \cite{W1972}, the composition $\xi: \C^\infty \to \gm^+ \to \cZ$
corresponding to the Harish-Chandra embedding is 
$$
\xi(Z) = (\exp(\widetilde{Z})(z_0) = \left [ \begin{smallmatrix}
	2\sqrt{-1}\, Z \\ 1 + \tr Z\, Z \\ \sqrt{-1}\,(1 - \tr Z\, Z)
	\end{smallmatrix} \right ] \in \cZ \subset \cP(V).
$$
Now $h(\xi(Z), \xi(Z)) = 
2Z^*\cdot 2Z - |1 + \tr Z\, Z|^2 - |1 - \tr Z\, Z|^2 < 0$, so
$$
h(\xi(Z),\xi(Z)) < 0 \Leftrightarrow 1 + |\tr Z\, Z|^2 -2 Z^*\, Z > 0
	\Leftrightarrow (1 - Z^*\, Z)^2 > (Z^*\, Z)^2 - |\tr Z\, Z|^2.
$$
Using $Z^*\, Z \geqq |\tr Z\, Z| \geqq 0$ we take positive square roots
to see
$$
\{Z \in \C^\infty \mid h(\xi(Z),\xi(Z)) < 0\}=D'_0 \cup D'_1 \text{ (disjoint)}
$$
where $D'_0$ is the nonempty bounded domain star shaped from $0$, 
$$
D'_0 = \{Z \in \C^\infty \mid 1 - Z^*\, Z > \bigl ( (Z^*\, Z)^2 - 
	|\tr Z\, Z|^2\bigr )^{1/2}\,\},
$$
and $D'_1$ is the nonempty unbounded domain star shaped from $\infty$,
$$
D'_1 = \{Z \in \C^\infty \mid Z^*\, Z - 1 > \bigl ( (Z^*\, Z)^2 - 
        |\tr Z\, Z|^2\bigr )^{1/2}\,\}.
$$
Using Witt's Theorem on the finite dimensional approximations,
$\xi^{-1}(D_0)$ is the topological component of 
$\{Z \in \C^\infty \mid h(\xi(Z),\xi(Z)) < 0\}$ containing $0$
so $D'_0 = \xi^{-1}(D_0)$.  We have proved
\begin{proposition}\label{realiz-so2inf}
The bounded symmetric domain $D_0$ for $SO(\infty,2)$ is given by
$$
\begin{aligned}
\xi^{-1}(D_0) &= \{Z \in \C^\infty\mid 1 - Z^*\, Z > \bigl ( (Z^*\, Z)^2 -
        |\tr Z\, Z|^2\bigr )^{1/2}\,\} \\ 
&= \{Z \in \C^\infty\mid 1 + |\tr Z\, Z|^2 -2 Z^*\, Z > 0
	\text{ and } Z^*\, Z < 1\}.
\end{aligned}
$$
\end{proposition}

Shortly we will see this in terms of partial Cayley transforms, but for
the moment we mention that $D'_1 = \xi^{-1}(D_1)$ where 
$z_1 = [e_1 - \sqrt{-1}\, e_2]$ and $D_1 = G_0(z_1)$ is given by
$$
\begin{aligned}
\xi^{-1}(D_1) &= \{Z \in \C^\infty\mid Z^*\, Z - 1 > \bigl ( (Z^*\, Z)^2 -
        |\tr Z\, Z|^2\bigr )^{1/2}\,\} \\
&= \{Z \in \C^\infty\mid 1 + |\tr Z\, Z|^2 -2 Z^*\, Z > 0
        \text{ and } Z^*\, Z  > 1\}.
\end{aligned}
$$

The action of $G_0$ on $D_0$ is somewhat complicated because of the
quadratic term $q: \C^\infty \to \C$ given by $q(Z) = \tr Z\, Z$.
If $Z \in \xi^{-1}(D_0)$ the $Z^*\, Z < 1$ so $|q(Z)| < 1$, and the
formula for $\xi(Z)$ says
$$
\text{if } q = q(Z) \ne 1 \text{ then } 
	\xi(Z) = (\exp(\widetilde{Z})(z_0) = \left [ \begin{smallmatrix}
	2\sqrt{-1}\,Z \\ 1 + q(Z) \\ 1 - \sqrt{-1}\,q(Z)
        \end{smallmatrix} \right ] \in \cZ \subset \cP(V).
$$
Now, by straightforward computation,
\begin{proposition}\label{action-soinf2}
The action $g(Z) = \xi^{-1}g\xi(Z)$ of $G_0$ on the open orbit $D_0$ is given by
$$
\text{if } g =  \left ( \begin{smallmatrix} A & B \\ C & D
        \end{smallmatrix} \right ) \text{ and } \left ( \begin{smallmatrix}
	Z_1 \\ Z_2 \end{smallmatrix} \right ) = \left ( \begin{smallmatrix}
	 2\sqrt{-1}\,Z \\ 1 + q(Z) \\ 1 - \sqrt{-1}\,q(Z)
	\end{smallmatrix} \right ) \text{ then } g(Z) =
	\tfrac{1}{(1,\sqrt{-1})(CZ_1 + DZ_2)} (AZ_1 + BZ_2).
$$
{\rm Here $\tfrac{1}{(1,\sqrt{-1})(CZ_1 + DZ_2)}$ is $1 \times 1$ and 
is viewed as a complex number.}
\end{proposition}

Express $V = V_+ \oplus V_-$ where $V_+ = \Span\{e_i \mid i < 0\}$ and
$V_- = \Span\{e_1,e_2\}$.  Then $h(V_+,V_-) = 0 = b(V_+,V_-)$.  
Let $[v] \in \cZ \subset \cP(V)$ with $G_0([v])$ open in $\cZ$, in other
words with $h(v,v) \ne 0$.  If $\pi_-(v) = a(e_1 + \sqrt{-1}\,e_2) +
b(e_1 - \sqrt{-1}\,e_2)$ then $0 = b(v,v) = 2ab$.  Replacing $v$ within $[v]$
now the only possibilities are (i) $\pi_-(v) = (e_1 + \sqrt{-1}\,e_2)$, 
(ii) $\pi_-(v) = (e_1 - \sqrt{-1}\,e_2)$ and (iii) $v \in V_+$\,.
The domains $D_0 = G_0([e_1 + \sqrt{-1}\,e_2])$ and 
$D_1 = G_0([e_1 - \sqrt{-1}\,e_2])$, so the possibilities (i) and (ii)
correspond to $D_0$ and $D_1$\,.  They are equivalent under complex conjugation
and each has signature $(0,1,0)$.
See Remark \ref{orientation}.  The bounded symmetric domains $D_0$ 
and $D_1$ are of tube type.  
\smallskip

The $G_0$--stabilizer of $V_+$\,, which
is $SO(\infty)\times SO(2)$, is transitive on the projective light cone 
in $V_+$\,; thus the possibility (iii) corresponds to the domain
$D_2 = G_0([e_{-1} + \sqrt{-1}\,e_{-2}])$, signature $(1,0,0)$.  This
completes the verification of
\begin{lemma}\label{three-open}
There are three open orbits for the action of $SO(\infty,2)$ on $\cZ$:
the two $h$--negative definite orbits $D_0 = G_0([e_1 + \sqrt{-1}\,e_2])$
and $D_1 = G_0([e_1 - \sqrt{-1}\,e_2])$, and the $h$--positive definite
orbit $D_2 = G_0([e_{-1} + \sqrt{-1}\,e_{-2}])$.
\end{lemma}

Now we do this more carefully with the partial Cayley transforms.  
Each $c_i(e_j) = e_j$ for $j \notin \{-2,-1,1,2\}$.  In the 
basis $\{e_{-2},e_{-1},e_1,e_2\}$ of
$\Span\{e_{-2},e_{-1},e_1,e_2\}$, the $c_i$ have matrices
\begin{equation}\label{cayleysoinf2}
c_1 = \tfrac{1}{\sqrt{2}} \left ( \begin{smallmatrix}  1 &  0 &  1 &  0 \\
                                                       0 &  1 &  0 &  1 \\
                                                      -1 &  0 &  1 &  0 \\
                                                       0 & -1 &  0 &  1
                          \end{smallmatrix} \right ) \text{ and }
c_2 = \tfrac{1}{\sqrt{2}} \left ( \begin{smallmatrix}  1 &  0 &  1 &  0 \\
                                                       0 &  1 &  0 & -1 \\
                                                      -1 &  0 &  1 &  0 \\
                                                       0 &  1 &  0 &  1
                          \end{smallmatrix} \right ).
\end{equation}
Thus there are six $G_0$--orbits on $\cZ$.  Their base points are
\begin{equation} \label{so2-orbits}
\begin{aligned}
z_{0,0} &= z_0 &&= [e_1 + \sqrt{-1}\, e_2]                &\text{ (negative)}\\
z_{0,1} &= c_1^2z_0 &&= [e_{-2} + \sqrt{-1}\, e_{-1}]     &\text{ (positive)}\\
z_{0,2} &= c_1^2c_2^2z_0 &&= [e_1 - \sqrt{-1}\, e_2]      &\text{ (negative)}\\
z_{1,1} &= c_1z_0 &&= [e_{-2} + \sqrt{-1}\, e_{-1} + e_1 + \sqrt{-1}\, e_2]
                                                         &\text{ (isotropic)}\\
z_{1,2} &= c_1c_2^2z_0 &&= [e_{-2}-\sqrt{-1}\, e_{-1} - e_1 + \sqrt{-1}\, e_2] 
                                                         &\text{ (isotropic)}\\
z_{2,2} &= c_1c_2z_0 &&= [e_{-2} + \sqrt{-1}\, e_2]       &\text{ (isotropic)}
\end{aligned}
\end{equation}
That gives 3 open orbits $D_0 = G_0(z_{0,0})$, $D_2 = G_0(z_{0,1})$ and
$D_1 = G_0(z_{0,2})$; it gives two intermediate orbits $G_0(z_{1,1})$
and $G_0(z_{1,2})$; and it gives one closed orbit $G_0(z_{2,2})$.

\subsection{Bounded Symmetric Domains for $SL(\infty;\mathbb{R})$ 
and $SL(\infty;\mathbb{H})$.}
\label{ssec4e}
\setcounter{equation}{0}

There are no complex bounded symmetric domains for these $SL(m;\F)$, 
$m < \infty$, except the unit disk in $\C$, corresponding to 
$SL(2;\R) \cong SU(1,1) \cong SL(1;\H)$.  In particular there is no 
complex bounded symmetric domain for $SL(\infty;\R)$,
and there is none for $SL(\infty;\H)$.

\section{Cycles and Cycle Spaces}\label{sec5}
\setcounter{equation}{0}

In the finite dimensional setting, where $D$ is an open $G_0$--orbit (flag
domain) in $\cZ = G/Q$, a maximal compact subgroup $K_0 \subset G_0$ has
just one orbit $Y$ on $D$ that is a complex submanifold \cite{W1969}.  The
$G$--translates of $Y$ that are contained in $D$ form the {\em cycle space}
$\cM_D$\,.  That cycle space is sometimes called the {\em universal domain} or
{\em crown} of the flag domain.  It has many uses in harmonic analysis and 
algebraic geometry; see \cite{FHW2005}.  It also has remarkable 
complex--geometric and function--analytic
properties; for example it is a contractible Stein manifold, it has an 
explicit geometric description, and it is the key ingredient for the
double fibration transform of which one special case is the Penrose
Transform.  Here we extend some of the basic results
on cycle spaces to an infinite dimensional setting.

\subsection{Basic Results.}\label{ssec5a}
\setcounter{equation}{0}

We fix an open $G_0$--orbit $D$ in the complex flag manifold $\cZ_{E,\cF}
\cong G/Q$ with $E$ as described in Section \ref{sec3} and $\cF$
compatible with $E$.  (The results of Section \ref{ssec3g} show when
there are open $G_0$--orbits in $\cZ_{E,\cF}$.) Let $K_0$ be a maximal 
lim--compact subgroup of $G_0$ and let $K \subset G$ be its 
complexification.  As in the
finite dimensional case \cite[Theorem 4.3.1]{FHW2005},

\begin{theorem}\label{base-cycle}
There is a unique orbit $K_0(z) \subset D$ such that
$K_0(z)$ is a complex submanifold of the flag manifold $\cZ_{\cF,E}$\,.
Further, $K \cap Q_z$ is a parabolic subgroup of $K$ and
$K_0(z) = K(z) \cong K/(K \cap Q_z)$, so
$K_0(z)$ is a complex flag manifold.
\smallskip

If $C \subset D$ is a lim--compact complex submanifold then the following are 
equivalent: {\rm (i)} $C$ is a $K_0$--orbit, {\rm (ii)} $C$ is a $K$--orbit,
and {\rm (iii)} $C = K_0(z)$.
\end{theorem}

\begin{proof}  The idea is to use the bases of Section \ref{sec3} together
with the results of \cite[Section 6]{DP2004} in order to take a direct limit
using the finite dimensional flag domain result of 
\cite[Theorem 4.3.1]{FHW2005}.
\smallskip

We run through the cases of Section \ref{sec3}.  In each case, the basis
$E$ of $V$  is a disjoint union of finite sets $E_\ell$ where (i) if
there is a hermitian form $h$ then the $\Span\{E_\ell\}$ are $h$--nondegenerate 
and mutually $h$--orthogonal, (ii) if there is a bilinear form $b$ then the 
$\Span\{E_\ell\}$ are $b$--nondegenerate and mutually $b$--orthogonal as 
well.  Further, we may assume that $\ell$ runs over the positive integers, 
\smallskip

Denote $\widetilde{E}_\ell = \bigcup_{k < \ell} E_k$ and
$V_\ell = \Span\{\widetilde{E}_\ell\}$.  In view of (i) and (ii) just
above, $G_0 = \varinjlim G_{\ell,0}$ where
$G_{\ell,0}$ = $\{ g \in G_0 \mid gV_\ell = V_\ell\}|_{V_\ell}$ is a 
finite dimensional real simple Lie group,
real form of the finite dimensional complex simple Lie group $G_\ell =
\{ g \in G \mid gV_\ell = V_\ell\}|_{V_\ell}$\,.  Further, 
$Q = \varinjlim Q_\ell$ where $Q_\ell$ is the
$G_\ell$--stabilizer of $\cF$\,.
\smallskip

We need a result of Dimitrov and Penkov
\cite[Proposition 6.1]{DP2004}.  They assume that $Q$
contains a splitting Cartan subgroup of $G$, but the argument is valid, as
in our case, when each $Q_\ell$ contains a splitting Cartan subgroup $H_\ell$
of $G_\ell$ with $H_\ell \subset H_m$ for $\ell \leqq m$\,.  Denote
$\widetilde{E}_\ell = \bigcup_{k \leqq \ell} E_k$\,.  Then
$\cZ_{\cF,E} = \varinjlim  \cZ_{\cF\cap V_\ell,\widetilde{E}_\ell}$ where 
we either eliminate or ignore repetitions in the $\cF\cap V_\ell$\,.
\smallskip

Since $D$ is open in $\cZ_{\cF,E}$\,, the flag $\cF$ is nondegenerate in $V$,
so by construction of the $\widetilde{E}_\ell$ each flag $\cF\cap V_\ell$
is nondegenerate in $V_\ell$\,.  Thus $G_{\ell,0}(\cF\cap V_\ell)$ is open
in $\cZ_{\cF\cap V_\ell,\widetilde{E}_\ell}$ for each $\ell$.  It follows
\cite[Theorem 2.12]{W1969} that, for each $\ell$, 
$Q_\ell$ contains a fundamental Cartan subgroup $T_{\ell,0}$ of $G_{\ell,0}$.
Any two fundamental Cartan subgroups of $G_{\ell,0}$ are conjugate, and
if $k \leqq \ell$ then any fundamental Cartan subgroup of $G_{k,0}$ is
contained in a fundamental Cartan subgroup of $G_{\ell,0}$.  Thus we may 
assume $T_{k,0} \subset T_{\ell,0}$ for $k \leqq \ell$.  
\smallskip

The fundamental Cartan $T_{\ell,0}$ determines a maximal compact subgroup
$K_{\ell,0}$ of $G_{\ell,0}$ such that $T_{\ell,0} \cap K_{\ell,0}$ is a
compact Cartan subgroup of $K_{\ell,0}$\,.  Let $K_\ell$ denote the
complexification of $K_{\ell,0}$\,.  Now $K_{k,0} \subset K_{\ell,0}$ 
and $K_k \subset K_\ell$ for $k \leqq \ell$.  Following 
\cite[Theorem 4.3.1]{FHW2005}, 
$K_{\ell,0}(\cF\cap V_\ell)$ is the unique $K_{\ell,0}$--orbit in
$G_{\ell,0}(\cF\cap V_\ell)$ that is a complex submanifold of 
$\cZ_{\cF\cap V_\ell,\widetilde{E}_\ell}$\,, and 
$K_{\ell,0}(\cF\cap V_\ell) = K_\ell(\cF\cap V_\ell)$.  
\smallskip

Suppose  for the moment that 
$K_0 = \varinjlim K_{\ell,0}$\,.  Then $K_0(\cF)$ is the unique $K_0$--orbit
in $D$ that is a complex submanifold of $\cZ_{\cF,E}$ and $K_0(\cF) = K(\cF)$.
Theorem \ref{base-cycle} follows for this particular maximal lim--compact
subgroup $K_0$ in $G_0$.  But any two maximal lim--compact subgroups of $G_0$
are conjugate, so Theorem \ref{base-cycle} follows for every choice of $K_0$\,.
\end{proof}

Let $K_0$ be the maximal lim--compact subgroup of $G_0$ constructed above 
in the proof of Theorem \ref{base-cycle}.  We will use the notation
\begin{equation}
Y = K_0(\cF), \,\,\, G\{Y\} = \{g \in G \mid gY \subset D\}, \,\,\, 
G_Y = \{g \in G \mid gY = Y\}, \,\,\, \cM'_D = G\{Y\}\cdot Y.
\end{equation}
where $Y$ is the complex $K_0$--orbit in the open $G_0$--orbit $D \subset
\cZ_{\cF,E}$\,.  We refer to $Y$ as the {\bf base cycle} in $D$.  Note that
the elements of $G\{Y\}$ do not have to map $D$ into itself; $G\{Y\}$ simply
is the set of of all elements in the complex group that keep the base cycle 
$Y$ inside $D$.  Further, $\cM'_D:= G\{Y\}\cdot Y$ is the set of all such
$G$--translates of $Y$.

\begin{lemma} \label{open-in-G}
$G\{Y\}$ is an open subset of $G$, $G_Y$ is a closed complex subgroup of $G$,
and $\cM'_D = G\{Y\}/G_Y$ is an open subset of the complex manifold
$G/G_Y$\,.  In particular $\cM'_D$ is an open complex submanifold of $G/G_Y$\,.
\end{lemma}

\begin{proof}
For each $\ell$,  $G_Y \cap G_\ell$ is a closed complex subgroup of $G_\ell$
and $G\{Y\} \cap G_\ell$ is an open subset of $G_\ell$\,.  It follows that
$G_Y$ is a closed complex subgroup of $G$ and $G\{Y\}$ is an open subset
of $G$.  Now $G\{Y\}/G_Y$ is open in the complex homogeneous space $G/G_Y$\,.
\end{proof}
The complex manifold structure of $\cM'_D$ specifies its topology, and we
define
\begin{definition}
Let $\cM_D$ denote the topological component of $Y$ in $\cM'_D$\,.  
Then $\cM_D$ is the {\bf cycle space} of the flag domain $D$.
\end{definition}

Note that $\cM_D$ is not always the same as the Barlet cycle space
\cite{M2004} from the theory of complex analytic spaces.  See 
\cite[Part IV]{FHW2005} for the comparison.  Next, we discuss several cases 
where we can really pin down the structure of $\cM_D$\,.
\smallskip

\subsection{Cycle Spaces for $SU(\infty,q)$, $q \leqq \infty$.} \label{ssec5b}
\setcounter{equation}{0}
In this section $G_0 = SU(\infty,q)$ and its maximal lim-compact subgroup is
$$
\begin{aligned}
&K_0 = S(U(\infty)\times U(q)) \,\,\,= {\lim}_{p\to\infty}\, S(U(p)\times U(q)) 
	\text{ if } q < \infty,\\
&K_0 = S(U(\infty)\times U(\infty)) = {\lim}_{r,s\to\infty}\, S(U(r)\times U(s))
	\text{ if } q = \infty.
\end{aligned}
$$
This corresponds to an $h$--orthogonal decomposition
\begin{equation}\label{su-splitting0}
\begin{aligned}
&\C^{\infty,q} \phantom{j}= V_+ \oplus V_- \text{ where }
        V_+ = \Span\{\dots , e_{-3}, e_{-2}, e_{-1}\},\,
        V_- = \Span\{e_1, \dots , e_q\} \text{ or }\\
&\C^{\infty,\infty} = V_+ \oplus V_- \text{ where }
        V_+ = \Span\{\dots , e_{-3}, e_{-2}, e_{-1}\},\,
        V_- = \Span\{e_1, e_2, e_3, \dots \}.
\end{aligned}
\end{equation}
Here we use
the related orthogonal basis $E$ given by (\ref{su-basis}) and the
hermitian form $h$ of (\ref{hform-su}) that defines $G_0$\,.
Let $\cF = (F_k)$ be a generalized flag in $V = \C^{\infty,q}$
that is weakly compatible with $E$.  Let $\cF^{(1)} \in \cZ_{\cF,E}$ so that
$D = G_0(\cF^{(1)})$ is an open $G_0$--orbit.  Then we may assume that
$\cF^{(1)}$ is compatible with our choice of $E$, so it fits the decomposition
(\ref{su-splitting0}) in the sense that
\begin{equation}\label{su-splitting1}
\cF^{(1)} = (F^{(1)}_k) \text{ where each }
        F^{(1)}_k = (F^{(1)}_k \cap V_+) \oplus (F^{(1)}_k \cap V_-).
\end{equation}
Then $K_0(\cF^{(1)})$ is the unique $K_0$--orbit in
$D$ that is a complex submanifold of the flag manifold $\cZ_{\cF,E}$\,.
Concretely, $K_0(\cF^{(1)})$ is the product of ``smaller'' complex
flag manifolds,
\begin{equation} \label{Y-splits}
\begin{aligned}
Y = Y_1 \times Y_2 &\text{ where }\\
   &Y_1 = K_0(\cF^{(1)} \cap V_+) = 
	U(\infty)(\cF^{(1)}\cap V_+) \text{ in } V_+ \text{ and }\\
   &Y_2 = K_0(\cF^{(1)} \cap V_-) = 
	U(q)(\cF^{(1)} \cap V_-) \text{ in } V_-
\end{aligned}
\end{equation}
where $\cF^{(1)} \cap V_+$ is the generalized flag of the  
$(F^{(1)}_k \cap V_+)$ and
$\cF^{(1)} \cap V_-$ is the generalized flag of the  
$(F^{(1)}_k \cap V_-)$, ignoring repetitions.
The signature sequence $\{(a_k,b_k)\}$, where $h$ has
signature $(a_k,b_k,0)$ on $F^{(1)}_k$, specifies the open orbit 
in $\cZ_{\cF,E}$ and the factors of $Y$.
\smallskip

As in the finite dimensional case, this shows that the $G$--translates of
$Y$ contained in $D$ correspond to the decompositions
$V = W' \oplus W''$ where (i) $W'$ is a maximal
positive definite subspace such that $V_+ \cap W'$ has finite codimension
in both in $V_+$ and in $W'$, and (ii) $W''$ is a maximal negative 
definite subspace such that $V_- \cap W''$ has finite codimension in 
both in $V_-$ and in $W''$.
If $\cF^{(1)} = \cF^{(1)} \cap V_+$ the correspondence depends only on 
$W'$, and if $\cF^{(1)} = \cF^{(1)} \cap V_-$ it depends only on $W''$.
Any two such decompositions $V = W' \oplus W''$ are $G$--equivalent.
\smallskip

\begin{definition}\label{def-bded-dom}
{\rm The {\sl positive bounded symmetric domain}
$\cB^+_E$ associated to $(V,E)$ is the space of all maximal positive
definite subspaces $W' \subset V$ such that $W'\cap V_+$ has finite
codimension in both $W'$ and $V_+$\,.  The {\sl negative
bounded symmetric domain} $\cB^-_E$ associated to $(V,E)$ is the
space of all maximal negative definite subspaces $W'' \subset V$ such that
$W'' \cap V_-$ has finite codimension in both $W''$
and $V_-$\,.}
\end{definition}

As constructed, each element $W' \in \cB^+_E$ is in the $G$--orbit of $V_+$\,.
Relative to the basis $E$ we look at $g = \left ( \begin{smallmatrix}
A & B \\ C & D \end{smallmatrix}\right ) \in G$ such that $gW' \in
\cB^+_E$, in other words such that the column span of $\left (
\begin{smallmatrix} A \\ C  \end{smallmatrix}\right )$ is positive definite.
The column span is preserved under right multiplication by $A$, so the
positive definite condition is 
$\left ( \begin{smallmatrix} I \\ -CA^{-1} \end{smallmatrix}\right )^*\cdot
\left ( \begin{smallmatrix} I \\ CA^{-1} \end{smallmatrix}\right ) \ggg 0$.
In other words $gW' \in
\cB^+_E$ simply means that $gW'$ is the column span of an infinite
matrix $\left ( \begin{smallmatrix} I \\ Z_1  \end{smallmatrix}\right )$
such that $I - Z_1^*Z_1 \ggg 0$.  Similarly $gW'' \in
\cB^-_E$ simply means that $gW''$ is the column span of an infinite
matrix $\left ( \begin{smallmatrix} Z_2 \\ I  \end{smallmatrix}\right )$
such that $I - Z_2Z_2^* \ggg 0$.  The distinction is that the $G$--stabilizer
of $V_+ \in \cB^+_E$ is the parabolic $P$ consisting of all $\left (
\begin{smallmatrix} A & B \\ 0 & D  \end{smallmatrix}\right )$,
while the $G$--stabilizer of $V_- \in \cB^-_E$ is the opposite parabolic
$\tr P = P^{opp}$ consisting of all 
$\left ( \begin{smallmatrix} A & 0 \\ C & D
\end{smallmatrix}\right )$.  Thus they have conjugate complex structures:
$\cB^+_E$ has holomorphic tangent space represented by the matrices
$\left ( \begin{smallmatrix} 0 & 0 \\ C & 0  \end{smallmatrix}\right ) \in \gg$
while the holomorphic tangent space of $\cB^-_E$ is represented by the
$\left ( \begin{smallmatrix} 0 & B \\ 0 & 0  \end{smallmatrix}\right ) \in \gg$.\smallskip

Reformulating this,
\begin{lemma} \label{bded-dom-su}
Suppose that $G_0 = SU(\infty,q)$, $q \leqq \infty $.  Then the
positive bounded symmetric domain associated to
the triple $(V,G_0,E)$ is
$\cB^+_E \cong \{Z_1 \in \C^{\infty \times q} \mid I - Z_1^*Z_1 \ggg 0\}$ in
$G/P$,  and the negative bounded symmetric domain for
$(V,G_0,E)$ is the complex conjugate domain
$\cB^-_E \cong \{Z_2 \in \C^{q\times \infty} \mid I - Z_2^*Z_2 \ggg 0\}$
in  $G/P^{opp}$.
\end{lemma}
The action of $G_0$ on these bounded symmetric domains is described in
Section \ref{ssec4a}.
\smallskip

Now we are ready to prove the following theorem.

\begin{theorem}\label{main-su}
Let $G_0 = SU(\infty,q)$ with $q \leqq \infty$.  Let $D$ be an open 
$G_0$--orbit $G(\cF^{(1)})$ in $\cZ_{\cF,E}$.  In the notation of
{\rm (\ref{su-splitting0})}, the positive definite bounded symmetric domain 
$\cB^+_E$ for $(V,G_0,E)$ is the set of
all positive definite $G$-translates of $V_+$ and the negative definite 
bounded symmetric domain $\cB^-_E$ for $(V,G_0,E)$ is the set of is the
set of all negative definite $G$-translates of $V_-$\,.  The
$\cB^\pm_E$ are antiholomorphically diffeomorphic, in other words each is 
the complex conjugate of the other.  There are three
cases for the structure of the cycle space, as follows.  

If every $F^{(1)}_k \in \cF^{(1)}$ is positive definite then $\cM_D$
is holomorphically diffeomorphic to $\cB^+_E$\,.

If every $F^{(1)}_k \in \cF^{(1)}$ is negative definite then $\cM_D$
is holomorphically diffeomorphic to $\cB^-_E$\,.

If some $F^{(1)}_k \in \cF^{(1)}$ is indefinite then $\cM_D$ is
holomorphically diffeomorphic to $\cB^+_E \times \cB^-_E$\,.
\end{theorem}

\begin{proof} Directly from Definition \ref{def-bded-dom}, $gV_+$ is 
$h$--positive definite if and only if $gV_+ \in \cB^+_E$\,, $gV_-$
is $h$--negative definite if and only if $gV_- \in \cB^-_E$\,, and
both properties hold for $gV_\pm$ if and only if $(gV_+, gV_-)
\in \cB^+_E \times \cB^-_E$\,. 

First suppose that $\cF^{(1)} = \cF^{(1)} \cap V_+$\,, $g \in G$
and $k \in K_0$\,.  Note $kV_+ = V_+$\,.
If $gV_+$ is positive definite then $gk(\cF^{(1)}) \in D$ because $gk(\cF^{(1)})$ is
nondegenerate and $D$ is the only open $G_0$--orbit in $\cZ_{\cF,E}$
consisting of positive definite subspaces.   Thus $gY \subset D$, in other
words $gY \in \cM'_D$\,.  Conversely if $gY \in \cM'_D$, so $gY \subset D$,
then $gY$ consists of positive definite subspaces.  If $0 \ne F^{(1)} \in \cF^{(1)}$
then $\Span K_0(F^{(1)}) = V_+$\,, so $\Span gY = gV_+$ is positive definite.
Now $gY \in \cM'_D$ if and only if $gV_+ \in \cB^+_E$\,.
\smallskip

Similarly, if $\cF^{(1)} = \cF^{(1)} \cap V_-$ and $g \in G$ then
$gY \in \cM'_D$ if and only if $gV_- \in \cB^-_E$\,.
\smallskip

In the general case $\cF^{(1)} \cap V_+ \ne \cF^{(1)} \ne \cF^{(1)} \cap V_-$ the arguments
just above show that $gV_+$ is positive definite if and only if 
$gK_0(\cF^{(1)} \cap V_+)$ consists of positive definite subspaces; and 
$gV_-$ is negative definite if and only if 
$gK_0(\cF^{(1)} \cap V_-)$ consists of negative definite subspaces.  
Thus $gY \in \cM'_D$ if and only if $gV_+$ is positive definite and
$gV_-$ is negative definite, in other words if and only if $(gV_+, gV_-)
\in \cB^+_E \times \cB^-_E$\,.
\smallskip

In all three cases we note that $\cM'_D$ is connected, so $\cM'_D = \cM_D$\,.
\smallskip

Finally, $h$--orthocomplementation is antiholomorphic and interchanges
$\cB^+_E$ with $\cB^-_E$\,.
\end{proof}

\subsection{Cycle Spaces for $Sp(\infty;\R)$.}\label{ssec5c}
\setcounter{equation}{0}
The case $G_0 = Sp(\infty;\R) = Sp(\infty;\C) \cap U(\infty,\infty)$
differs from the $SU(\infty,q)$ cases mainly in that we use $b$--isotropic 
flags where $b$ is the antisymmetric bilinear form that defines
$Sp(\infty;\C)$.  Specifically, we use the basis and forms described in 
Section \ref{sec3}, given by (\ref{so*-basis}) and (\ref{spr-forms}),
where $b$ defines $Sp(\infty;\C)$ and $h$ defines $U(\infty,\infty)$.
\smallskip

The maximal lim-compact subgroups $K_0$ of $G_0 = Sp(\infty;\R)$ is
the $U(\infty)$ constructed as follows.  Relative to $h$,
\begin{equation}\label{sp-splitting}
V = V_+ \oplus V_- \text{ where }
        V_+ = \Span\{\dots , e_{-3}, e_{-2}, e_{-1}\} \text{ and }
        V_- = \Span\{e_1, e_2, e_3, \dots \}.
\end{equation}
The maximal lim-compact subgroup of $U(\infty,\infty)$ is
$U(V_+) \times U(V_-) = U(\infty) \times U(\infty)$, and $K_0$
is the subgroup 
$G_0 \cap \bigl ( U(\infty) \times U(\infty) \bigr ) \cong U(\infty)$.
In the ordered basis $\{e_{-1}, e_{-2}, \dots ; e_1, e_2, \dots\}$ it
would be diagonally embedded in $U(\infty) \times U(\infty)$.
\smallskip

Let $\cF$ be a $b$--isotropic generalized flag compatible with $E$.  
Let $\cF^{(1)} \in \cZ_{\cF,E}$ so that $D = G_0(\cF^{(1)})$ is open in 
$\cZ_{\cF,E}$.  Again, we may assume that $\cF^{(1)}$ is compatible with $E$,
in other words, it fits the splitting (\ref{sp-splitting}) in the sense that
$$
\cF^{(1)} =(F_k^{(1)}) \text{ where each } F^{(1)}_k = (F^{(1)}_k \cap V_+)\oplus (F^{(1)}_k \cap V_-).
$$
In particular $\cF^{(1)}$ is $h$--nondegenerate, corresponding to the fact that
$D = G_0(\cF)$ is open in $\cZ_{\cF,E}$, and $K_0(\cF^{(1)})$ is the unique 
$K_0$--orbit in $D$ that is a complex submanifold of $\cZ_{\cF,E}$\,.
\begin{lemma}\label{Y-splits-sp}
Define $\cF^{(1)} \cap V_+ = (F^{(1)}_k \cap V_+)$ and 
$\cF^{(1)}\cap V_-=(F^{(1)}_k \cap V_-)$, and spaces
$W_+ = \bigcup_k (F^{(1)}_k \cap V_+)$ and $W_- = \bigcup_k (F^{(1)}_k \cap V_-)$.
Then the complex lim--compact group orbit $Y = K_0(\cF^{(1)})$ is the 
subvariety of 
$$
\begin{aligned}
\widetilde{Y} = Y_1 \times Y_2 \text{ where }
   &Y_1 = K_0(\cF^{(1)} \cap V_+) = U(\infty)(\cF^{(1)}\cap V_+) \text{ in } 
	V_+ \text{ and }\\
   &Y_2 = K_0(\cF^{(1)} \cap V_-) = U(q)(\cF^{(1)}\cap V_-) \text{ in } V_-
\end{aligned}
$$
defined by $b(k_1W_+,k_2W_-) = 0$ for $k_1, k_2 \in K_0$\,.
The signature sequence $\{(a_k,b_k)\}$, where $h$ has
signature $(a_k,b_k,0)$ on $F^{(1)}_k$, specifies the open orbit 
in $\cZ_{\cF,E}$ and the factors of $\widetilde{Y}$.
\end{lemma}

\begin{proof} The projections $r_1: K_0 \to U(V_+)$ and $r_2:K_0 \to U(V_-)$
are isomorphisms.  Define $\mu: V \to V$ by $\mu(e_i) = e_{-i}$ and 
$\mu(e_{-i}) = -e_i$ for $i > 0$.  Since $\cF^{(1)}$ is $b$--isotropic and
compatible with $E$, each $F^{(1)}_k$ is spanned by a subset $S_k \subset E$
that never contains a pair $\{e_i,e_{-i}\}$.  Thus
each $(\cF_k^{(1)} \cap V_+) + \mu(\cF_k^{(1)} \cap V_+)$ is $b$--nondegenerate and
$h$--nondegenerate, and is orthogonal to 
$(\cF_k^{(1)} \cap V_-) + \mu(\cF_k^{(1)} \cap V_-)$ relative to both $b$ and $h$.
Now the action of $r_1(K_0)$ on 
$(\cF_k^{(1)} \cap V_+)+ \mu(\cF_k^{(1)} \cap V_+)$ and the action of $r_2(K_0)$
on $(\cF_k^{(1)} \cap V_-) + \mu(\cF_k^{(1)} \cap V_-)$ only involve disjoint 
subsets of $S_k \cup -S_k$.  Thus $Y \subset \widetilde{Y}$ and
$b(k_1W_+,k_2W_-) = 0$ for $k_1, k_2 \in K_0$\,.  
\smallskip

Conversely, if $(k_1(\cF_k^{(1)} \cap V_+),
k_2(\cF_k^{(1)} \cap V_+)) \in Y$, so it has form $(k(\cF_k^{(1)} \cap V_+),
k(\cF_k^{(1)} \cap V_+))$, then $k_1W_+ = kW_+ \perp_b kW_- = k_2W_-$.
Given $kW_+ \perp_b kW_-$\,, $K_0$ moves $(\cF_k^{(1)} \cap V_+)$ freely within
$V_+\cap (W_-)^\perp$ and moves $(\cF_k^{(1)} \cap V_-)$ freely within 
$V_+\cap (W_+)^\perp$\,.  That proves the first assertion.  The signature
sequence assertion is contained in Theorem \ref{open-orbit-su}.
\end{proof}

These considerations show that the $G$--translates of
$Y$ contained in $D$ correspond to decompositions
$V = W' \oplus W''$ where (i) $W'$ and $W''$ are maximal $b$--isotropic
subspaces of $V$, (ii) $W'$ is a maximal $h$--positive definite subspace 
such that $W' \cap V_+$ has finite codimension in both $W'$ and $V_+$\,, and
(iii) $W''$ is a maximal $h$--negative definite subspace such that
$W'' \cap V_-$ has finite codimension in both $W''$ and $V_-$\,.  
If $\cF^{(1)} = \cF^{(1)} \cap V_+$ the correspondence depends only on $W'$, and if
$\cF^{(1)} = \cF^{(1)} \cap V_-$ it depends only on $W''$.
Any two such decompositions $V = W' \oplus W''$ are $G$--equivalent.

\begin{definition}{\rm The {\sl positive bounded symmetric domain}
$\cB^+_E$ associated to $(V,b,E)$ is the space of all maximal 
$b$--isotropic $h$--positive definite subspaces $W' \subset V$ such that 
$W'\cap V_+$ has finite codimension in both $W'$ and $V_+$\,.  The {\sl negative
bounded symmetric domain} $\cB^-_E$ associated to $(V,b,E)$ is the
space of all maximal 
$b$--isotropic $h$--negative definite subspaces $W'' \subset V$ such that
$W'' \cap V_-$ has finite codimension in both $W''$ and $V_-$\,.}
\end{definition}

As constructed, each element $W' \in \cB^+_E$ is in the $G$--orbit of $V_+$\,.
Relative to the basis $E$ we look at $g = \left ( \begin{smallmatrix}
A & B \\ C & D \end{smallmatrix}\right ) \in G$ such that $gW' \in
\cB^+_E$, in other words such that the column span of $\left (
\begin{smallmatrix} A \\ C  \end{smallmatrix}\right )$ is $h$--positive 
definite.  That span is $b$--isotropic by definition of $G$, and the
column span is preserved under right multiplication by $A$.  Let 
$Z_1 = CA^{-1}$.  Then the $h$--positive definite condition is 
$\left ( \begin{smallmatrix} I \\ Z_1 \end{smallmatrix}\right )^*\cdot 
\left ( \begin{smallmatrix} I \\ Z_1 \end{smallmatrix}\right ) \ggg 0$, in
other words $I - Z_1^*Z_1 \ggg 0$.  Let $Z_1 = \left ( z_{i,j}\right )$ where
$i,j > 0$.  The column span of 
$\left ( \begin{smallmatrix} I \\ Z_1 \end{smallmatrix}\right )$ has basis
consisting of the $z_j := e_{-j} + \sum_{i > 0} z_{i,j}e_i$\,,  Compute 
$b(z_j,z_\ell) = z_{j,\ell} - z_{\ell,j}$.  So the $b$--isotropic condition is
$Z_1 = \tr{Z_1}$\,.
In other words $gW' \in
\cB^+_E$ simply means that $gW'$ is the column span of an infinite
matrix $\left ( \begin{smallmatrix} I \\ Z_1  \end{smallmatrix}\right )$
such that $I - Z_1^*Z_1 \ggg 0$ and $Z_1$ is symmetric.
\smallskip

Similarly $gW'' \in \cB^-_E$ simply means that $gW''$ is the column span 
of an infinite
matrix $\left ( \begin{smallmatrix} Z_2 \\ I  \end{smallmatrix}\right )$
such that $I - Z_2Z_2^* \ggg 0$ and $Z_2$ is symmetric.  The distinction 
is that the $G$--stabilizer of $V_+ \in \cB^+_E$ is the parabolic $P$ 
consisting of all $\left (
\begin{smallmatrix} A & B \\ 0 & D  \end{smallmatrix}\right )$ in $\gg$
while the $G$--stabilizer of $V_- \in \cB^-_E$ is the opposite parabolic
$\tr P = P^{opp}$ consisting of all 
$\left ( \begin{smallmatrix} A & 0 \\ C & D \end{smallmatrix}\right )$ in
$\gg_\C$\,.  Thus they have conjugate complex structures:
$\cB^+_E$ has holomorphic tangent space represented by the matrices
$\left ( \begin{smallmatrix} 0 & 0 \\ C & 0 \end{smallmatrix}\right )$ 
with $C$ symmetric while the holomorphic tangent space of $\cB^-_E$ is 
represented by the
$\left ( \begin{smallmatrix} 0 & B \\ 0 & 0 \end{smallmatrix}\right )$ 
with $B$ symmetric.
\smallskip

Reformulating this,
\begin{lemma} \label{bded-dom-sp}
Let $G_0 = Sp(\infty;\R)$.  Then the
positive bounded symmetric domain associated to
the triple $(V,G_0,E)$ is
$\cB^+_E \cong \{Z_1 \in \C^{\infty \times \infty} \mid I - Z_1^*Z_1 \ggg 0
\text{ and } Z_1 = {\tr{Z_1}}\}$ in
$G/P$,  and the negative bounded symmetric domain for
$(V,G_0,E)$ is the complex conjugate domain
$\cB^-_E \cong \{Z_2 \in \C^{\infty \times \infty} \mid I - Z_2^*Z_2 \ggg 0
\text{ and } Z_1 = {\tr{Z_1}}\}$ in  $G/P^{opp}$.
\end{lemma}
The action of $G_0$ on these bounded symmetric domains is described in
Section \ref{ssec4b}.
\smallskip

In any $K_0$--invariant Riemannian metric on $\widetilde{Y}$, $Y_1$ and $Y_2$
are the factors in the de Rham decomposition. The spaces 
$k(F_\ell^{(1)} \cap V_+)$ of the elements of $Y_1$ generate $V_+$ 
(or are zero), so either $Y$ determines $Y_1$ determines $V_+$\,, or 
the $F_\ell^{(1)} \cap V_+ = 0$. Similarly either $Y$ determines $V_-$\,, or 
the $F_\ell^{(1)} \cap V_- = 0$.
Now apply $g^{-1}$ whenever  $g \in G\{Y\}$ to see that $gY$ determines 
$gV_+$ or $gV_-$ or both, as appropriate.  Exactly as in the proof of
Theorem \ref{main-su} we arrive at the following structure theorem. 

\begin{theorem}\label{main-sp}
Let $G_0 = Sp(\infty;\R)$ and let $D$ be an open 
$G_0$--orbit $G(\cF^{(1)})$ in $\cZ_{\cF,E}$.  In the notation of
{\rm (\ref{sp-splitting})}, the positive definite bounded symmetric domain 
$\cB^+_E$ for $(V,G_0,E)$ is the set of
all positive definite $G$-translates of $V_+$ and the negative definite 
bounded symmetric domain $\cB^-_E$ for $(V,G_0,E)$ is the set of is the
set of all negative definite $G$-translates of $V_-$\,.  The
$\cB^\pm_E$ are antiholomorphically diffeomorphic, in other words each is 
the complex conjugate of the other.  There are three
cases for the structure of the cycle space, as follows.  

If every $F^{(1)}_k \in \cF^{(1)}$ is positive definite then $\cM_D$
is holomorphically diffeomorphic to $\cB^+_E$\,.

If every $F^{(1)}_k \in \cF^{(1)}$ is negative definite then $\cM_D$
is holomorphically diffeomorphic to $\cB^-_E$\,.

If some $F^{(1)}_k \in \cF^{(1)}$ is indefinite then $\cM_D$ is
holomorphically diffeomorphic to $\cB^+_E \times \cB^-_E$\,.
\end{theorem}

\subsection{Cycle Spaces for $SO^*(\infty)$.}\label{ss6d}
\setcounter{equation}{0}
The case $SO^*(\infty) = SO(\infty;\C) \cap U(\infty,\infty)$ is very similar
to the case of $Sp(\infty;\R)$.  The main difference is that the bilinear
form $b$ is symmetric rather than antisymmetric.  Concretely, we have
$$
\begin{aligned}
&E = \{\dots , e_{-k}, e_{-k+1}, \dots , e_{-1}; e_1, \dots , e_{k-1}, e_k ,
        \dots \}\text{, ordered basis of } V; \\
&b(e_i,e_j) = \delta_{i+j,0}\,, h(e_i,e_j) = \delta_{i,j} \text{ for } i < 0
	\text{ and } h(e_i,e_j) = -\delta_{i,j} \text{ for } i > 0.
\end{aligned}
$$
Again we use the $h$--orthogonal splitting
\begin{equation}\label{so-splitting}
V = V_+ \oplus V_- \text{ where }
        V_+ = \Span\{\dots , e_{-3}, e_{-2}, e_{-1}\} \text{ and }
        V_- = \Span\{e_1, e_2, e_3, \dots \}.
\end{equation}
The maximal lim-compact subgroup of $U(\infty,\infty)$ is
$U(V_+) \times U(V_-) = U(\infty) \times U(\infty)$.  Exactly as in
the $Sp(\infty;\R)$ case, $K_0$ is the subgroup
$G_0 \cap \bigl ( U(\infty) \times U(\infty) \bigr ) \cong U(\infty)$.
In the ordered basis $\{e_{-1}, e_{-2}, \dots ; e_1, e_2, \dots\}$ it
would be diagonally embedded in $U(\infty) \times U(\infty)$.
\smallskip

Let $\cF$ be a $b$--isotropic generalized flag compatible with $E$.  Let
$\cF^{(1)} \in \cZ_{\cF,E}$ such that $D = G_0(\cF^{(1)})$ is open in
$\cZ_{\cF,E}$.  We may assume that $\cF$ is compatible with $E$, so
$$
\cF^{(1)}=(F_k^{(1)}) \text{ where each } F^{(1)}_k = (F^{(1)}_k \cap V_+)\oplus (F^{(1)}_k \cap V_-).
$$
In particular $\cF^{(1)}$ is $h$--nondegenerate and $K_0(\cF^{(1)})$ is the unique
$K_0$--orbit in $D$ that is a complex submanifold of $\cZ_{\cF,E}$\,.
With no nontrivial change, the proof of Lemma \ref{Y-splits-sp} also proves
\begin{lemma}\label{Y-splits-so}
Define $\cF^{(1)} \cap V_+ = (F^{(1)}_k \cap V_+)$ and 
$\cF^{(1)}\cap V_-=(F^{(1)}_k \cap V_-)$, and spaces
$W_+ = \bigcup_k (F^{(1)}_k \cap V_+)$ and $W_- = \bigcup_k (F^{(1)}_k \cap V_-)$.
Then the complex lim--compact group orbit $Y = K_0(\cF^{(1)})$ is the subvariety
of
$$
\begin{aligned}
\widetilde{Y} = Y_1 \times Y_2 \text{ where }
   &Y_1 = K_0(\cF^{(1)} \cap V_+) = U(\infty)(\cF^{(1)}\cap V_+) \text{ in }
        V_+ \text{ and }\\
   &Y_2 = K_0(\cF^{(1)} \cap V_-) = U(q)(\cF^{(1)}\cap V_-) \text{ in } V_-
\end{aligned}
$$
defined by $b(k_1W_+,k_2W_-) = 0$ for $k_1, k_2 \in K_0$\,.
The signature sequence $\{(a_k,b_k)\}$, and $($where relevant -- see
{\rm Remark \ref{orientation}}$)$ the orientation, 
specifies the open orbit in $\cZ_{\cF,E}$ and the factors of $\widetilde{Y}$.
\end{lemma}

Now the $G$--translates of
$Y$ contained in $D$ correspond to decompositions
$V = W' \oplus W''$ where (i) $W'$ and $W''$ are maximal $b$--isotropic
subspaces of $V$, (ii) $W'$ is a maximal $h$--positive definite subspace
such that $W' \cap V_+$ has finite codimension in both $W'$ and $V_+$\,, and
(iii) $W''$ is a maximal $h$--negative definite subspace such that
$W'' \cap V_-$ has finite codimension in both $W''$ and $V_-$\,.
If $\cF^{(1)} = \cF^{(1)} \cap V_+$ the correspondence depends only on $W'$, and if
$\cF^{(1)} = \cF^{(1)} \cap V_-$ it depends only on $W''$.
Any two such decompositions $V = W' \oplus W''$ are $G$--equivalent.

\begin{definition}{\rm The {\sl positive bounded symmetric domain}
$\cB^+_E$ associated to $(V,b,E)$ consists of all maximal
$b$--isotropic $h$--positive definite subspaces $W' \subset V$ such that
$W'\cap V_+$ has finite codimension in both $W'$ and $V_+$\,.  The {\sl negative
bounded symmetric domain} $\cB^-_E$ associated to $(V,b,E)$ consists of 
all maximal
$b$--isotropic $h$--negative definite subspaces $W'' \subset V$ such that
$W'' \cap V_-$ has finite codimension in both $W''$ and $V_-$\,.}
\end{definition}

Let  $g = \left ( \begin{smallmatrix}
A & B \\ C & D \end{smallmatrix}\right ) \in G$ relative
to $E$, such that $gW' \in \cB^+_E$\,.  Computing as for $Sp(\infty;\R)$, 
with $Z_1 = CA^{-1}$, we see that $gW' \in \cB^+_E$ if and only if
$gW'$ is the column span of an infinite
matrix $\left ( \begin{smallmatrix} I \\ Z_1  \end{smallmatrix}\right )$
such that $I - Z_1^*Z_1 \ggg 0$ and $Z_1$ is antisymmetric.
Similarly $gW'' \in \cB^-_E$ if and only if $gW''$ is the column span
of an infinite
matrix $\left ( \begin{smallmatrix} Z_2 \\ I  \end{smallmatrix}\right )$
such that $I - Z_2Z_2^* \ggg 0$ and $Z_2$ is antisymmetric.  
The $G$--stabilizer of $V_+ \in \cB^+_E$ is the parabolic $P$
consisting of all $\left (
\begin{smallmatrix} A & B \\ 0 & D  \end{smallmatrix}\right )$ in $\gg_\C$,
and the $G$--stabilizer of $V_- \in \cB^-_E$ is the opposite parabolic
$\tr P = P^{opp}$ consisting of all
$\left ( \begin{smallmatrix} A & 0 \\ C & D \end{smallmatrix}\right )$ in
$\gg_\C$\,.  Thus they have conjugate complex structures:
$\cB^+_E$ has holomorphic tangent space represented by the matrices
$\left ( \begin{smallmatrix} 0 & 0 \\ C & 0 \end{smallmatrix}\right )$ 
with $C$ antisymmetric while the holomorphic tangent space of $\cB^-_E$ 
is represented by the matrices $\left ( \begin{smallmatrix} 0 & B \\ 0 & 0 
\end{smallmatrix}\right )$ with $B$ antisymmetric.
\smallskip

Reformulating this,
\begin{lemma} \label{bded-dom-so}
Let $G_0 = SO^*(\infty)$.  Then the
positive bounded symmetric domain associated to
the triple $(V,G_0,E)$ is
$\cB^+_E \cong \{Z_1 \in \C^{\infty \times \infty} \mid I - Z_1^*Z_1 \ggg 0
\text{ and } Z_1 + {\tr{Z_1}} = 0\}$ in
$G/P$,  and the negative bounded symmetric domain for
$(V,G_0,E)$ is the complex conjugate domain
$\cB^-_E \cong \{Z_2 \in \C^{\infty q\times \infty} \mid I - Z_2^*Z_2 \ggg 0
\text{ and } Z_1 + {\tr{Z_1}} = 0\}$ in  $G/P^{opp}$.
\end{lemma}
The action of $G_0$ on these bounded symmetric domains is described in
Section \ref{ssec4c}.
\smallskip

Arguing just as for Theorems \ref{main-su} and \ref{main-sp}, we arrive 
at the following structure theorem.

\begin{theorem}\label{main-so}
Let $G_0 = SO^*(\infty)$ and let $D$ be an open
$G_0$--orbit $G(\cF^{(1)})$ in $\cZ_{\cF,E}$.  In the notation of
{\rm (\ref{so-splitting})}, the positive definite bounded symmetric domain
$\cB^+_E$ for $(V,G_0,E)$ is the set of
all positive definite $G$-translates of $V_+$ and the negative definite
bounded symmetric domain $\cB^-_E$ for $(V,G_0,E)$ is the set of is the
set of all negative definite $G$-translates of $V_-$\,.  The
$\cB^\pm_E$ are antiholomorphically diffeomorphic, in other words each is
the complex conjugate of the other.  There are three
cases for the structure of the cycle space, as follows.

If every $F^{(1)}_k \in \cF^{(1)}$ is positive definite then $\cM_D$
is holomorphically diffeomorphic to $\cB^+_E$\,.

If every $F^{(1)}_k \in \cF^{(1)}$ is negative definite then $\cM_D$
is holomorphically diffeomorphic to $\cB^-_E$\,.

If some $F^{(1)}_k \in \cF^{(1)}$ is indefinite then $\cM_D$ is
holomorphically diffeomorphic to $\cB^+_E \times \cB^-_E$\,.
\end{theorem}

\subsection{Cycle Spaces for $SO(\infty,2)$.}\label{ss6e}
\setcounter{equation}{0}
Now we come to the rather delicate case $G_0 = SO(\infty,2)$, where
the lim--compact dual of the complex bounded symmetric domain is a
nondegenerate quadric in a complex projective space.  We specify $G_0$
by the basis (\ref{so2-basis}) and the forms (\ref{so2-forms}).
Let $$
\begin{aligned}
&V_{even} = 
\Span(\{e_{2k} + \sqrt{-1}e_{2k+1} \mid k < 0\}\cup\{e_1 - \sqrt{-1}e_2\}),\\
&V_{odd} = 
\Span(\{e_{2k} - \sqrt{-1}e_{2k+1} \mid k < 0\}\cup\{e_1 + \sqrt{-1}e_2\}).
\end{aligned}
$$
They are maximal $b$--isotropic subspaces of $V$,
paired by $b(e_j + \sqrt{-1}e_{j+1}, e_j - \sqrt{-1}e_{j+1}) = 2$.
This basis $E$ leads to the same splitting of $V$ as the one based on 
(\ref{su-basis}):

\begin{equation}\label{splitso2}
V = V_+ \oplus V_- \text{ where }
	V_+ = \Span\{\dots , e_{-3}, e_{-2}, e_{-1}\}  \text{ and }
        V_- = \Span\{e_1, e_2\}. 
\end{equation}
We denote
$$
\cP^\infty \text{ is the projective space } \cP(V) \text{ and }
\cZ \text{ is the quadric } b(v,v) = 0 \text{ in } \cP^\infty\,.
$$
The maximal lim-compact subgroup of $G_0$ is
$K_0 = SO(V_+) \times SO(V_-) = SO(\infty) \times SO(2)$.  The
complex $K_0$--orbits within the open $G_0$--orbits on $\cZ$ (from Lemma 
\ref{three-open} and (\ref{so2-orbits}) are 
\begin{equation}\label{so2-orbits-cpt}
\begin{aligned}
&\text{in } D_0 = G_0([e_1 + \sqrt{-1}e_2]): \phantom{Xi}
  K_0([e_1 + \sqrt{-1}e_2]) = \text{ (single point $[e_1 + \sqrt{-1}e_2]$),} \\
&\text{in } D_1 = G_0([e_1 - \sqrt{-1}e_2]): \phantom{Xi}
  K_0([e_1 - \sqrt{-1}e_2]) = \text{ (single point $[e_1 - \sqrt{-1}e_2]$),} \\
&\text{in } D_2 = G_0([e_{-2} + \sqrt{-1}e_{-1}]): 
  K_0([e_{-2} + \sqrt{-1}e_{-1}]) = \cZ \cap \cP(V_+) \text{ quadric in } \cP(V_+).
\end{aligned}
\end{equation}

\begin{definition}{\rm The {\sl positive bounded symmetric domain}
$\cB^+_{E'}$ associated to $(V,b,E')$ consists of all maximal
$b$--isotropic $h$--positive definite subspaces $W' \subset V$ such that
$W'\cap V_+$ has finite codimension in both $W'$ and $V_+$\,.  Those subspaces 
have codimension $2$ in $V$.  The {\sl negative bounded symmetric domain} 
$\cB^-_{E'}$ associated to $(V,b,E')$ consists of all maximal
$b$--isotropic $h$--negative definite subspaces $W'' \subset V$.
(Since $\dim W'' = 2 = \dim V_-$ the finite codimension condition is
automatic.)}
\end{definition}

Now more generally let $\cF = (F_k)$ be an isotropic generalized flag in $V$
that is weakly compatible with $E'$.  Let $\cF^{(1)} \in \cZ_{\cF,E'}$ for which
$D = G_0(\cF^{(1)})$ is an open $G_0$--orbit.  We may assume that
$\cF^{(1)}$ is compatible with our choice of $E'$, so it fits the decomposition
(\ref{so-splitting}) as before:
\begin{equation}\label{so-splitting1}
\cF^{(1)} = (F^{(1)}_k) \text{ where each }
        F^{(1)}_k = (F^{(1)}_k \cap V_+) \oplus (F^{(1)}_k \cap V_-).
\end{equation}
Then $K_0(\cF^{(1)})$ is the unique $K_0$--orbit in
$D$ that is a complex submanifold of the flag manifold $\cZ_{\cF,E'}$\,.
Somewhat trivially, $K_0(\cF^{(1)})$ is the product of ``smaller'' complex
flag manifolds,
\begin{equation} \label{Y-splits-so2}
\begin{aligned}
Y = Y_1 \times Y_2 &\text{ where }\\
   &Y_1 = K_0(\cF^{(1)} \cap V_+) = SO(\infty)(\cF^{(1)}\cap V_+) \text{ in } V_+ \text{ and }\\
   &Y_2 = K_0(\cF^{(1)} \cap V_-) = SO(2)(\cF^{(1)}\cap V_-) \text{ in } V_-
\end{aligned}
\end{equation}
where $\cF^{(1)} \cap V_+ = ((F^{(1)}_k \cap V_+))$ and
$\cF^{(1)}\cap V_-=((F^{(1)}_1 \cap V_-))$.
The signature sequence $\{(a_k,b_k)\}$, where $h$ has
signature $(a_k,b_k,0)$ on $F^{(1)}_k$, specifies the open orbit in $\cZ_{\cF,E'}$
and the factors of $Y$.
\smallskip

If $\cF^{(1)} = \cF^{(1)}\cap V_+$, in other words $D = D_2$ and the cycles are of the
form $K_0(gV_+)$ with $g \in G$, then $M_D$ consists of the maximal
$b$--isotropic $h$--positive definite subspaces of $V$.
If $\cF^{(1)} = \cF^{(1)}\cap V_-$, in other words $D = D_0$ or $D = D_1$ and the
cycles are single points, then then $M_D$ consists of the maximal
$b$--isotropic $h$--negative definite subspaces of $V$.  If 
$\cF^{(1)}\cap V_+ \ne \cF^{(1)} \ne \cF^{(1)}\cap V_-$ then $M_D$ is the product.  Thus

\begin{lemma} \label{bded-dom-so2}
Let $G_0 = SO(\infty,2)$.  Then the
positive bounded symmetric domain associated to
the triple $(V,G_0,E')$ is
$\cB^+_{E'} \cong \{Z \in \C^\infty \mid 1 + |\tr{Z}Z|^2 -2Z^*Z > 0
\text{ and } Z^*Z < 1 \}$ in
$G/P$,  and the negative bounded symmetric domain for
$(V,G_0,E')$ is the complex conjugate domain
$\cB^-_{E'} \cong \{Z \in \C^\infty \mid 1 + |\tr{Z}Z|^2 -2Z^*Z > 0
\text{ and } Z^*Z > 1 \}$ in $G/P^{opp}$.
\end{lemma}
The action of $G_0$ on these bounded symmetric domains is described in
Section \ref{ssec4d}.
\smallskip

The argument for Theorem \ref{main-su} remains valid here, with one
small modification.  Recall Lemma \ref{three-open} and (\ref{so2-orbits}).
There is just one open orbit $D_2 = G_0([e_{-1}+ \sqrt{-1}\,e_{-2}])$
consisting of $h$--positive definite subspaces, but there are two orbits,
$D_0 = G_0([e_1 + \sqrt{-1}\,e_2])$ and $D_1 = G_0([e_1 - \sqrt{-1}\,e_2])$, 
consisting of negative definite subspaces.  These last two are related by 
complex conjugation of $V$ over the real span of $E$.  Suppose that $D$ is 
either $D_0$ or $D_2$\,, that $\cF^{(1)} = \cF^{(1)} \cap V_-$\,, and that $gV_-$ 
is negative definite.  Then $gY \subset (D_0 \cup D_1)$. As $gY$ is connected,
either $gY \subset D_0$ or $gY \subset D_1$\,.  Thus $gY \in \cM'_D$\,,
and $gY \in \cM_D$ just when $gY \subset D$.  With this adjustment the
proof of Theorem \ref{main-su} holds here, and the result is 

\begin{theorem}\label{main-so2}
Let $G_0 = SO(\infty,2)$ and let $D$ be an open
$G_0$--orbit $G(\cF^{(1)})$ in $\cZ_{\cF,E'}$.  In the notation of
{\rm (\ref{so-splitting})}, the positive definite bounded symmetric domain
$\cB^+_{E'}$ for $(V,G_0,E')$ is the set of
all positive definite $G$-translates of $V_+$ and the negative definite
bounded symmetric domain $\cB^-_E$ for $(V,G_0,E')$ is the set of is the
set of all negative definite $G$-translates of $V_-$\,.  The
$\cB^\pm_{E'}$ are antiholomorphically diffeomorphic, in other words each is
the complex conjugate of the other.  There are three
cases for the structure of the cycle space, as follows.

If every $F^{(1)}_k \in \cF^{(1)}$ is positive definite then $\cM_D$
is holomorphically diffeomorphic to $\cB^+_{E'}$\,.

If every $F^{(1)}_k \in \cF^{(1)}$ is negative definite then $\cM_D$
is holomorphically diffeomorphic to $\cB^-_{E'}$\,.

If some $F^{(1)}_k \in \cF^{(1)}$ is indefinite then $\cM_D$ is
holomorphically diffeomorphic to $\cB^+_{E'} \times \cB^-_{E'}$\,.
\end{theorem}

\section{Real and Quaternionic Domains and Cycle Spaces} \label{sec6}
\setcounter{equation}{0}
In Section \ref{sec4} we worked out the structure of finitary complex bounded 
symmetric domains, and in Section \ref{sec5} we applied those results to 
obtain the structure of cycle spaces on corresponding flag domains.  In this
section we develop a variation on those results for
particular real and quaternionic flag manifolds and cycle spaces
based on the groups $SO(\infty,q)$ and $Sp(\infty,q)$, $q \leqq \infty$. 
Those groups provide real and quaternionic analogs of the complex domains
of $SU(\infty,q)$.  The methods and results are similar to
those of Section \ref{ssec4a}, Section \ref{ssec5b}, and the last part
of \cite{SW1975}.

\subsection{The Real Bounded Symmetric Domain for $SO(\infty,q)$.}
\setcounter{equation}{0}
\label{ssec6a}
In Section \ref{ssec4a} we looked at the bounded domain of maximal
negative definite subspaces of $(V,h)$ contained in $\cZ_{\cF,E}$\,, 
where $V$ has basis $E$ given 
by (\ref{su-basis}) and where the hermitian form $h$ is given by 
(\ref{hform-su}).  We studied it as an $SU(\infty,q)$--orbit on
the complex Grassmann manifold of $q$--dimensional subspaces of $V$
weakly compatible with $E$.  Here we look at the real analog, the
(real -- not complex) bounded symmetric domain of maximal
negative definite subspaces of $(V_0,b)$ where $V_0$ is the real span
of $E$ and the symmetric bilinear form $b$ is the restriction of $h$
to $V_0$\,.  Then we use it to describe real cycle spaces for open 
orbits on the corresponding real flag manifolds.
\smallskip

We consider the real group $G_0 = SO(\infty,q)$, $q \leqq \infty$ and the flag
$\cF = (0,F,V_0)$ where $F = \Span_\R\{e_i \mid i > 0\}$.  View $G_0$ as
a closed subgroup of $G := SL(\infty + q; \R)$.
That gives us the {\em real} flag manifold 
\begin{equation}\label{real-grass}
\gX_{\cF,E} = \{\text{ subspaces } F^{(1)} \subset V_0
	\mid (0,F^{(1)},V) \text{ is } E\text{--commensurable to } \cF\}
	= G(\cF)
\end{equation}
where the second equality follows as in the argument of Lemma \ref{wk-compat}.
Note that $\gX_{\cF,E}$ is a real Grassmann manifold.  The domain of
interest to us in this context is
\begin{equation}\label{bdedsym-so}
D_0 = \{\cF^{(1)} = (0,F^{(1)},V_0) \in \gX_{\cF,E} \mid F^{(1)}
        \text{ maximal negative definite subspace of } V_0\}.
\end{equation}
If $\tau: V \to V$ denotes complex conjugation of $V$ over $V_0$ then the
domain $D_0$ of (\ref{bdedsym-so}) can be identified with the fixed point
set of $\tau$ on the complex Grassmannian of Section \ref{ssec4a}.
\smallskip

We use the $b$--orthogonal decomposition
$V_0 = (V_0)_+ \oplus (V_0)_-$ where $(V_0)_+ = \Span_\R\{e_i \mid i < 0\}$
and $(V_0)_- = \Span_\R\{e_i \mid i > 0\}$.  Consider the corresponding
$b$--orthogonal projections $\pi_\pm$\,.  The kernel of $\pi_-$ is 
$b$--positive definite
so it has zero intersection with $F^{(1)}$ for any $\cF^{(1)} = 
(0,F^{(1)},V_0) \in D_0$.
Thus $\pi_-: F^{(1)} \cong (V_0)_-$ is injective, and it is surjective as
well because $F^{(1)}$ is a maximal negative definite subspace.  Now we 
have a well defined linear map
\begin{equation}\label{def-so-infq}
X_{F^{(1)}}: (V_0)_- \to (V_0)_+ \text{ defined by } \pi_-(x) \mapsto \pi_+(x)
        \text{ for } x \in F^{(1)}.
\end{equation}

As $\cF^{(1)}$ is weakly compatible with $E$, the matrix of $X_{F^{(1)}}$ 
relative to $E$ has only finitely many nonzero entries, i.e. $X_{F^{(1)}}$
is finitary.  Further, $\pi_-: F^{(1)} \cong (V_0)_-$ defines an $\R$--basis
$\{e''_i\}$ of $F^{(1)}$ by $\pi_-(e''_i)
= e_i$\,.  Write $e''_i = e_i + \sum_{j<0} x_{j,i}e_j$; then
$(x_{j,i})$ is the matrix of $X_{F^{(1)}}$.  The fact that $F^{(1)}$ is $b$--negative
definite, translates to the
matrix condition $I - {\tr (x_{j,i})}\,(x_{j,i}) \ggg 0$, equivalently the
operator condition $I - {\tr X_{F^{(1)}}}\, X_{F^{(1)}} \ggg 0$.  Conversely if
$X: (V_0)_- \to (V_0)_+$ is finitary and satisfies $I - {\tr X}\,X \ggg 0$,
then the real column span of its matrix relative to $E$ is a maximal negative
definite subspace $F^{(1)}$, and $\cF^{(1)} = (0,F^{(1)},V_0) \in D_0$\,.
\smallskip

The same computation as in Section \ref{ssec4a} shows that the
block form matrices of elements of $G_0$ act by
$\left ( \begin{smallmatrix} A & B \\ C & D \end{smallmatrix}\right ) :
\left ( \begin{smallmatrix} X \\ I \end{smallmatrix}\right ) \to
\left ( \begin{smallmatrix} AX+B \\ CX+D \end{smallmatrix}\right )$,
which has the same real column span as
$\left ( \begin{smallmatrix} (AX+B)(CX+D)^{-1} \\ I \end{smallmatrix}\right )$.
So $G_0$ acts by 
$\left ( \begin{smallmatrix} A & B \\ C & D \end{smallmatrix}\right ) : X
\to (AX+B)(CX+D)^{-1}$. In summary, 

\begin{proposition}\label{bded-realization-so-infq}
$D_0$ is realized as the bounded domain of all finitary
$X: (V_0)_- \to (V_0)_+$ such that $I - {\tr X}\, X \ggg 0$,  and there the 
action of 
$G_0$ is $\left ( \begin{smallmatrix} A & B \\ C & D \end{smallmatrix}\right ) 
: X \to (AX+B)(CX+D)^{-1}$.
\end{proposition}

Again, there are $q+1$ open $G_0$--orbits on $\gX_{\cF,E}$ corresponding to
nondegenerate signatures:
$$
\begin{aligned}
D_k = G_0(0,F_{(k)},V_0) \text{ where } &F_{(k)} =
        {\Span_\R}\{e_{-k}, \dots , e_{-1}; e_{k+1}, \dots , e_q\}
                \text{ if } q < \infty, \\
&F_{(k)} = {\Span_\R}\{e_{-k}, \dots , e_{-1}; e_{k+1}, e_{k+2}, \dots\}
                \text{ if } q = \infty, \\
\end{aligned}
$$
More generally the $G_0$--orbits on $\gX_{\cF,E}$
of signature $(a,b,c) = (\pos,\neg,\nul)$ have $a$ and $c$ finite and
$\leqq q$.  We denote them by
\begin{equation}
\begin{aligned}
D_{a,b,c} &= G_0(0,(F_+ + F_- + F_0),V) \text{ where }\\
        F_0& = \Span_\R\{e_{-c}+e_c, \dots ,e_{-1}+e_1\} \text{ (null) }\\
        F_+& = \Span_\R\{e_{-c-a}, \dots , e_{-c-1}\} \text{ (positive)}\\
        F_-& = \Span_\R\{e_{c+1}, \dots , e_{c+b}\}, \,\,  q < \infty;\,\,
                \Span_\R\{e_{c+1},e_{c+2}, \dots\}, \,\, q = \infty
                        \text{ (negative)}.
\end{aligned}
\end{equation}
As in the complex case, the open orbits are the $D_a = D_{a,b,0}\,, 
a < \infty$ and $a+b = q$, i.e. the ones for $c = 0$.  If $q < \infty$ 
there is a unique closed orbit, 
$D_{0,0,q} = \{(0,F^{(1)},V_0) \in \gX_{\cF,E} \mid
b(F^{(1)},F^{(1)}) = 0\}$; it is in the
closure of every orbit. If $q = \infty$ there is no closed orbit.
\smallskip

The Cayley transforms are given by (\ref{cayley-su}): $c_k(e_j) = e_j$ if
$j \ne \pm k$ and, in the basis $\{e_{-k},e_k\}$ of $\Span_\R\{e_{-k},e_k\}$,
$c_k$ has matrix $\tfrac{1}{\sqrt{2}} \left ( \begin{smallmatrix}
1 & 1 \\ -1 & 1 \end{smallmatrix} \right )$. This sends real 
subspaces of $V$ to real subspaces; that is why, in Section \ref{ssec4a}, 
we based (\ref{cayley-su}) on the one variable Cayley transform that sends 
$0 \to 1 \to \infty \to -1 \to 0$ and maps the unit disk to the right half 
plane.  As a riemannian symmetric space, the real
Grassmannian $\gX_{\cF,E}$ has rank $q$.  Just as in the complex case
the $G_0$--orbits on $\gX_{\cF,E}$ are the $G_0(c_1\dots c_s c_{s+1}^2\dots
c_{s+t}^2\cF)$, and the open ones are those for which $s = 0$.  If $q < \infty$
then $G_0(c_1\dots c_q\cF)$ is the closed orbit, and if $q = \infty$ then there
is no closed $G_0$--orbit on $\gX_{\cF,E}$.
\smallskip

The maximal lim-compact subgroup of $G_0$ is
$$
\begin{aligned}
&K_0 = SO(\infty) \times SO(q) \,\,\,= \bigl ({\lim}_{p\to\infty} SO(p)\bigr ) 
	\times SO(q) \text{ if } q < \infty,\\
&K_0 = SO(\infty) \times SO(\infty) = {\lim}_{p,q\to\infty} 
	\bigl (SO(p) \times SO(q)\bigr ) \text{ if } q = \infty.
\end{aligned}
$$
This corresponds to the $b$--orthogonal decomposition
$\R^{\infty,q} \phantom{j}= (V_0)_+ \oplus (V_0)_-$\,.
Let $\cF = (F_k)$ be a generalized flag in $V = \R^{\infty,q}$
that is weakly compatible with $E$.  Let $\cF^{(1)} \in \gX_{\cF,E}$ so that
$D = G_0(\cF^{(1)})$ is an open $G_0$--orbit.  Then we may assume that
$\cF^{(1)}$ is compatible with our choice of $E$, so it fits the decomposition
$\R^{\infty,q} \phantom{j}= (V_0)_+ \oplus (V_0)_-$ in the sense that
\begin{equation}\label{so-splitting1-real}
\cF^{(1)} = (F^{(1)}_k) \text{ where each }
        F^{(1)}_k = (F^{(1)}_k \cap (V_0)_+) \oplus (F^{(1)}_k \cap (V_0)_-).
\end{equation}
Then $K_0(\cF^{(1)})$ is the real analog -- in fact a real form -- of  
the base cycle in the complexification of $D$.
Concretely, $K_0(\cF^{(1)})$ is the product of ``smaller'' real
flag manifolds,
\begin{equation} \label{quat-Y-splits}
\begin{aligned}
Y = Y_1 \times Y_2 &\text{ where }\\
   &Y_1 = K_0(\cF^{(1)} \cap (V_0)_+) = SO(\infty)(\cF^{(1)}\cap (V_0)_+) \text{ in } (V_0)_+ \text{ and }\\
   &Y_2 = K_0(\cF^{(1)} \cap (V_0)_-) = SO(q)(\cF^{(1)}\cap (V_0)_-) \text{ in } (V_0)_-\,.
\end{aligned}
\end{equation}
where 
$$
\begin{aligned}
&\cF^{(1)} \cap (V_0)_+ = ((F^{(1)}_1 \cap (V_0)_+)\subset\dots\subset 
	(F^{(1)}_n \cap (V_0)_+)), \\
&\cF^{(1)}\cap (V_0)_-=((F^{(1)}_1 \cap (V_0)_-)\subset\dots\subset (F^{(1)}_n \cap (V_0)_-)).
\end{aligned}
$$
The signature sequence $\{(a_k,b_k)\}$, where $h$ has
signature $(a_k,b_k,0)$ on $F^{(1)}_k$, specifies the open orbit in $\gX_{\cF,E}$
and the factors of $Y$.
\smallskip

This shows that the $G$--translates of
$Y$ contained in $D$ correspond to the decompositions
$V_0 = W_0' \oplus W_0''$ where (i) $W_0'$ is a maximal
positive definite subspace such that $(V_0)_+ \cap W_0'$ has finite codimension
in both $W_0'$ and $(V_0)_+$\,, and (ii) $W_0''$ is a maximal negative 
definite subspace such that $(V_0)_- \cap W_0''$ has finite codimension in 
both $W_0''$ and $(V_0)+$\,.
If $\cF^{(1)} = \cF^{(1)} \cap (V_0)_+$ the correspondence depends only on $W_0'$, and if
$\cF^{(1)} = \cF^{(1)} \cap (V_0)_-$ it depends only on $W_0''$.
Any two such decompositions $V_0 = W_0' \oplus W_0''$ are $G$--equivalent.
\smallskip

\begin{definition}{\rm The {\sl positive real bounded symmetric domain}
$\cB^+_E$ associated to $(V_0,b,E)$ is the space of all maximal positive
definite subspaces $W_0' \subset V_0$ such that $W_0'\cap (V_0)_+$ has finite
codimension in both $W_0'$ and $(V_0)_+$\,.  The {\sl negative
bounded symmetric domain} $\cB^-_E$ associated to $(V_0,b,E)$ is the
space of all maximal negative definite subspaces $W_0'' \subset V$ such that
$W_0'' \cap (V_0)_-$ has finite codimension in both $W_0''$
and $(V_0)_-$\,.}
\end{definition}

As constructed, each element $W_0' \in \cB^+_E$ is in the $G$--orbit of $(V_0)_+$.
Relative to the basis $E$ we look at $g = \left ( \begin{smallmatrix}
A & B \\ C & D \end{smallmatrix}\right ) \in G$ such that $gW_0' \in
\cB^+_E$, in other words such that the column span of $\left (
\begin{smallmatrix} A \\ C  \end{smallmatrix}\right )$ is positive definite.
The column span is preserved under right multiplication by $A$, so the
positive definite condition is 
$\tr{\bigl ( \begin{smallmatrix} I \\ -CA^{-1} \end{smallmatrix}\bigr )} \cdot
\bigl ( \begin{smallmatrix} I \\ CA^{-1} \end{smallmatrix}\bigr ) \ggg 0$.
In other words $gW_0' \in
\cB^+_E$ simply means that $gW_0'$ is the column span of an infinite real
matrix $\left ( \begin{smallmatrix} I \\ X_1  \end{smallmatrix}\right )$
such that $I - {\tr{X_1}}X_1 \ggg 0$.  Similarly $gW_0'' \in
\cB^-_E$ simply means that $gW_0''$ is the column span of an infinite real
matrix $\left ( \begin{smallmatrix} X_2 \\ I  \end{smallmatrix}\right )$
such that $I - {\tr{X_2}}X_2 \ggg 0$.  The $G$--stabilizer
of $0 \in \cB^+_E$ is the parabolic $P$ consisting of all $\left (
\begin{smallmatrix} A & B \\ 0 & D  \end{smallmatrix}\right )$,
while the $G$--stabilizer of $0 \in \cB^-_E$ is the opposite parabolic
$\tr P = P^{opp}$ consisting of all 
$\left ( \begin{smallmatrix} A & 0 \\ C & D
\end{smallmatrix}\right )$.  
Reformulating this,

\begin{lemma} \label{bded-dom-so-real}
Suppose that $G_0 = SO(\infty,q)$, $q \leqq \infty $.  Then the real
positive bounded symmetric domain associated to $(V,G_0,E)$ is
$\cB^+_E \cong \{X_1 \in \R^{\infty \times q} \mid I-{\tr{X_1}}X_1 \ggg 0\}$ in
$G/P$,  and the negative real bounded symmetric domain for
$(V,G_0,E)$ is 
$\cB^-_E \cong \{X_2 \in \R^{q\times \infty} \mid I - {\tr{X_2}}X_2 \ggg 0\}$
in  $G/P^{opp}$.
\end{lemma}
The action of $G_0$ on these bounded symmetric domains is linear fractional,
as described in Section \ref{ssec4a} for the complex case.
The proof of Theorem \ref{main-su} is valid here, giving us
the following structure theorem.

\begin{theorem}\label{main-so-real}
Let $G_0 = SO(\infty,q)$ with $2 < q \leqq \infty$.  Let $D$ be an open 
$G_0$--orbit $G((0,F^{(1)},V_0`)$ in the real flag manifold $\gX_{\cF,E}$.  
Then the positive definite bounded symmetric domain 
$\cB^+_E$ for $(V,G_0,E)$ is the set of
all positive definite $G$-translates of $(V_0)_+$ and the negative definite 
bounded symmetric domain $\cB^-_E$ for $(V,G_0,E)$ is the set of is the
set of all negative definite $G$-translates of $(V_0)_-$\,.  The
$\cB^\pm_E$ are diffeomorphic. There are three
cases for the structure of the cycle space, as follows.  

If every space $F^{(1)}_k \in \cF^{(1)}$ is positive definite then $\cM_D$
is diffeomorphic to $\cB^+_E$\,.

If every space $F^{(1)}_k \in \cF^{(1)}$ is negative definite then $\cM_D$
is diffeomorphic to $\cB^-_E$\,.

If some space $F^{(1)}_k \in \cF^{(1)}$ is indefinite then $\cM_D$ is
diffeomorphic to $\cB^+_E \times \cB^-_E$\,.
\end{theorem}

\subsection{The Quaternionic Bounded Symmetric Domain for $Sp(\infty,q)$.}
\setcounter{equation}{0}
\label{ssec6b}

We now look at the quaternionic analog of Section \ref{ssec6a}.
For that, we consider a quaternionic vector space $V_\H = \H^{\infty,q}$,  
one of whose underlying complex structures is that of $V = \C^{\infty,2q}$.
We look at the bounded symmetric domain of maximal negative definite
quaternionic subspaces of $(V_\H,h)$.   As suggested by Section \ref{ssec3c},
the complex basis $E$ of $V$ is replaced by an $\H$--basis
\begin{equation}
\begin{aligned}
L = &\{\dots , v_{-2}, v_{-1}; v_1, v_2, \dots , v_q\} \text{ for } q < \infty, \\
L = &\{\dots , v_{-2}, v_{-1}; v_1, v_2, v_3,\dots \} \text{ for } q = \infty.
\end{aligned}
\end{equation}
The relation with $E$ is $v_i = e_{2i}$ for $i < 0$ and 
$v_j = e_{2j-1}$ for $j > 0$.  The $\H$--hermitian form $h$ is
defined by $h(v_i,v_j) = \delta_{i,j}$ for $i < 0$ and $h(v_i,v_j) = 
-\delta_{i,j}$ for $i > 0$.
\smallskip

The real group is $G_0 = Sp(\infty,q)$, $q \leqq \infty$. We view
$G_0$ as a closed subgroup of the quaternionic linear 
group $G := SL(\infty + q; \H)$.  The flag
is $\cF = \{F\}$ where $F = \Span_\H\{e_i \mid i > 0\}$.
That gives us the {\em quaternionic} flag manifold 
\begin{equation}\label{quat-grass}
\gX_{\cF,L} = \{\text{subspaces } F^{(1)} \subset V_\H
	\mid (0,F^{(1)}, V_\H) \text{ is } L\text{--commensurable to } \cF\}
	= G(\cF)
\end{equation}
where the second equality follows as in the argument of Lemma \ref{wk-compat}.
Note that $\gX_{\cF,L}$ is a quaternionic Grassmann manifold.  The domain of
interest to us in this context is

\begin{equation}\label{bdedsym-sp}
D_0 = \{(0,F^{(1)},V_\H) \in \gX_{\cF,L} \mid F^{(1)}
        \text{ is a maximal $h$--negative definite subspace of } V_\H\}.
\end{equation}

Now consider the $h$--orthogonal decomposition
$V_\H = (V_\H)_+ \oplus (V_\H)_-$ where $(V_\H)_+$ denotes 
$\Span_\H\{e_i \mid i < 0\}$
and $(V_\H)_-$ denotes $\Span_\H\{e_i \mid i > 0\}$.  Consider the 
corresponding orthogonal 
projections $\pi_+: V_\H \to (V_\H)_+$ and $\pi_- : V_\H \to (V_\H)_-$\,, 
The kernel of $\pi_-$ is 
$h$--positive definite so it has zero intersection with $F^{(1)}$ for any 
$\cF^{(1)} = (0,F^{(1)},V_\H) \in D_0$.  Thus $\pi_-: F^{(1)} \cong (V_\H)_-$ 
is injective.  Since $F^{(1)}$ is a maximal $h$--negative definite subspace 
$\pi_-: F^{(1)} \cong (V_\H)_-$ is surjective as well.  Now we
have a well defined $\H$--linear map
\begin{equation}\label{def-sp-infq}
X_{F^{(1)}}: (V_\H)_- \to (V_\H)_+ \text{ defined by } \pi_-(x) \mapsto \pi_+(x)
        \text{ for } x \in F^{(1)}.
\end{equation}
As $\cF^{(1)}$ is weakly compatible with $L$, the matrix of $X_{F^{(1)}}$ 
relative to $L$ has only finitely many nonzero entries, i.e. $X_{F^{(1)}}$
is finitary.  Using $\pi_-: F^{(1)} \cong V_{\H,-}$ defines an $\H$--basis
$\{v''_i\}$ of $F^{(1)}$ by $\pi_-(v''_i)
= v_i$\,.  Write $v''_i = v_i + \sum_{j<0} x_{j,i}v_j$; then
$(x_{j,i})$ is the matrix of $X_{F^{(1)}}$.  The fact that $F^{(1)}$ is $h$--negative
definite, translates to the
matrix condition $I - (x_{j,i})^*\,(x_{j,i}) \ggg 0$, equivalently the
operator condition $I - X_{F^{(1)}}^*\, X_{F^{(1)}} \ggg 0$.  Conversely if
$X: (V_\H)_- \to (V_\H)_+$ is finitary and satisfies $I - X^*\,X \ggg 0$,
then the quaternionic column span of its matrix relative to $L$ is a maximal 
negative definite subspace $F^{(1)}$, and 
$\cF^{(1)} = (0,F^{(1)},V_\H) \in D_0$\,.
\smallskip

The same computation as in Section \ref{ssec4a} shows that the
block form matrices of elements of $G_0$ act by
$\left ( \begin{smallmatrix} A & B \\ C & D \end{smallmatrix}\right ) :
\left ( \begin{smallmatrix} X \\ I \end{smallmatrix}\right ) \to
\left ( \begin{smallmatrix} AX+B \\ CX+D \end{smallmatrix}\right )$,
which has the same quaternionic column span as
$\left ( \begin{smallmatrix} (AX+B)(CX+D)^{-1} \\ I \end{smallmatrix}\right )$.
So $G_0$ acts by the linear fractional
$\left ( \begin{smallmatrix} A & B \\ C & D \end{smallmatrix}\right ) : X
\to (AX+B)(CX+D)^{-1}$. In summary, 

\begin{proposition}\label{bded-realization-sp-infq}
$D_0$ is realized as the bounded domain of all finitary
$X: (V_\H)_- \to (V_\H)_+$ such that $I -  X^*\, X \ggg 0$,  and there the 
action of 
$G_0$ is $\left ( \begin{smallmatrix} A & B \\ C & D \end{smallmatrix}\right ) 
: X \to (AX+B)(CX+D)^{-1}$.
\end{proposition}

Again, there are $q+1$ open $G_0$--orbits on $\gX_{\cF,L}$ corresponding to
nondegenerate signatures:
$$
\begin{aligned}
D_k = G_0(0,F_{(k)},V_\H) \text{ where } &F_{(k)} =
        {\Span_\H}\{v_{-k}, \dots , v_{-1}; v_{k+1}, \dots , v_q\}
                \text{ if } q < \infty, \\
&F_{(k)} = {\Span_\H}\{v_{-k}, \dots , v_{-1}; v_{k+1}, v_{k+2}, \dots\}
                \text{ if } q = \infty, \\
\end{aligned}
$$
More generally the $G_0$--orbits on $\gX_{\cF,L}$
of signature $(a,b,c) = (\pos,\neg,\nul)$ have $a$ and $c$ finite and
$\leqq q$.  We denote them by
\begin{equation}
\begin{aligned}
D_{a,b,c} &= G_0(0,(F_+ + F_- + F_0),V_\H) \text{ where }\\
        &F_0 = \Span_\H\{v_{-c}+v_c, \dots ,v_{-1}+v_1\} \text{ (null) }\\
        &F_+ = \Span_\H\{v_{-c-a}, \dots , v_{-c-1}\} \text{ (positive)}\\
        &F_- = \Span_\H\{v_{c+1}, \dots , v_{c+b}\} \text{ if } q < \infty,\\
        &\phantom{F_- =i }\Span_\H\{v_{c+1},v_{c+2}, \dots\} 
		\text{ if } q = \infty \text{ (negative)}.
\end{aligned}
\end{equation}
The open orbits are the $D_a = D_{a,b,0}\,, 
a < \infty$ and $a+b = q$, i.e. the ones for $c = 0$.  If $q < \infty$ 
there is a unique closed orbit, 
$D_{0,0,q} = \{(0,F^{(1)},V_\H) \in \gX_{\cF,L} \mid
h(F^{(1)},F^{(1)}) = 0\}$; it is in the
closure of every orbit. If $q = \infty$ there is no closed orbit.
\smallskip

The Cayley transforms are given by (\ref{cayley-su}): $c_k(v_j) = v_j$ if
$j \ne \pm k$ and, in the basis $\{v_{-k},v_k\}$ of $\Span_\H\{v_{-k},v_k\}$,
$c_k$ has matrix $\tfrac{1}{\sqrt{2}} \left ( \begin{smallmatrix}
1 & 1 \\ -1 & 1 \end{smallmatrix} \right )$. This sends quaternionic 
subspaces of $V$ to quaternionic subspaces.
As a riemannian symmetric space, the quaternion
Grassmannian $\gX_{\cF,L}$ has rank $q$.  Just as in the complex case
the $G_0$--orbits on $\gX_{\cF,L}$ are the $G_0(c_1\dots c_s c_{s+1}^2\dots
c_{s+t}^2\cF)$, and the open ones are those for which $s = 0$.  If $q < \infty$
then $G_0(c_1\dots c_q\cF)$ is the closed orbit, and if $q = \infty$ then there
is no closed $G_0$--orbit on $\gX_{\cF,L}$.
\smallskip

The maximal lim-compact subgroup of $G_0$ is
$$
\begin{aligned}
&K_0 = Sp(\infty) \times Sp(q) \,\,\,= \bigl ({\lim}_{p\to\infty} Sp(p) \bigr )
	\times Sp(q) \text{ if } q < \infty,\\
&K_0 = Sp(\infty) \times Sp(\infty) = {\lim}_{p,q\to\infty} 
	\bigl ( Sp(p) \times Sp(q) \bigr ) \text{ if } q = \infty.
\end{aligned}
$$
This corresponds to the $h$--orthogonal decomposition
$\H^{\infty,q} \phantom{j}= (V_\H)_+ \oplus (V_\H)_-$.
Let $\cF = (F_k)$ be a generalized flag in $V = \H^{\infty,q}$
that is weakly compatible with $L$.  Let $\cF^{(1)} \in \gX_{\cF,L}$ so that
$D = G_0(\cF^{(1)})$ is an open $G_0$--orbit.  Then we may assume that
$\cF^{(1)}$ is compatible with our choice of $L$, so it fits the decomposition
$\H^{\infty,q} \phantom{j}= (V_\H)_+ \oplus (V_\H)_-$ in the sense that
\begin{equation}\label{so-splitting1-quat}
\cF^{(1)} = (F^{(1)}_k) \text{ where each }
        F^{(1)}_k = (F^{(1)}_k \cap (V_\H)_+) \oplus (F^{(1)}_k \cap (V_\H)_-).
\end{equation}
Then $K_0(\cF^{(1)})$ is the quaternionic analog -- in fact a quaternion form 
-- of  the base cycle when the latter is viewed as a quaternionic manifold.
Concretely, $K_0(\cF^{(1)})$ is the product of ``smaller'' quaternionic
flag manifolds,
\begin{equation} \label{quaternionic-Y-splits}
\begin{aligned}
Y = Y_1 \times Y_2 &\text{ where }\\
   &Y_1 = K_0(\cF^{(1)} \cap (V_\H)_+) = Sp(\infty)(\cF^{(1)}\cap (V_\H)_+) \text{ in } (V_\H)_+ \text{ and }\\
   &Y_2 = K_0(\cF^{(1)} \cap (V_\H)_-) = Sp(q)(\cF^{(1)}\cap (V_\H)_-) \text{ in } (V_\H)_-
\end{aligned}
\end{equation}
where 
$$
\begin{aligned}
&\cF^{(1)} \cap (V_\H)_+ = ((F^{(1)}_1 \cap (V_\H)_+)\subset\dots\subset 
	(F^{(1)}_n \cap (V_\H)_+)), \\
&\cF^{(1)}\cap (V_\H)_-=((F^{(1)}_1 \cap (V_\H)_-)\subset\dots\subset 
	(F^{(1)}_n \cap (V_\H)_-)).
\end{aligned}
$$
The signature sequence $\{(a_k,b_k)\}$, where $h$ has
signature $(a_k,b_k,0)$ on $F^{(1)}_k$, specifies the open orbit in $\gX_{\cF,L}$
and the factors of $Y$.
\smallskip

This shows that the $G$--translates of
$Y$ contained in $D$ correspond to the decompositions
$V_H\ = W_\H' \oplus W_\H''$ where (i) $W_\H'$ is a maximal
positive definite $\H$--subspace such that $(V_\H)_+ \cap W_\H'$ has finite 
codimension in both  $(V_\H)_+$ and $\cap W_\H'$\,, and (ii) $W_\H''$ is a 
maximal negative definite subspace such that $(V_\H)_- \cap W_\H''$ has 
finite codimension in both $(V_\H)_-$ and $W_\H''$\,.
If $\cF^{(1)} = \cF^{(1)} \cap (V_\H)_+$ the correspondence depends only on 
$W_\H'$, and if $\cF^{(1)} = \cF^{(1)} \cap (V_\H)_-$ it depends only on 
$W_\H''$.
Any two such decompositions $V_\H = W_\H' \oplus W_\H''$ are $G$--equivalent.
\smallskip

\begin{definition}{\rm The {\sl positive quaternionic bounded symmetric domain}
$\cB^+_L$ associated to $(V_\H,b,L)$ is the space of all maximal positive
definite subspaces $W_\H' \subset V_\H$ such that $W_\H'~\cap~(V_\H)_+$ has finite
codimension in both $W_\H'$ and $(V_\H)_+$\,.  The {\sl negative quaternionic
bounded symmetric domain} $\cB^-_L$ associated to $(V_\H,b,L)$ is the
space of all maximal negative definite subspaces $W_\H'' \subset~V_\H$ such that
$W_\H''~\cap (V_\H)_-$ has finite codimension in both $W_\H''$
and $(V_\H)_-$\,.}
\end{definition}

As constructed, each element $W_\H' \in \cB^+_L$ is in the $G$--orbit of $(V_\H)_+$.
Relative to the basis $L$ we look at $g = \left ( \begin{smallmatrix}
A & B \\ C & D \end{smallmatrix}\right ) \in G$ such that $gW_\H' \in
\cB^+_L$, in other words such that the column span of $\left (
\begin{smallmatrix} A \\ C  \end{smallmatrix}\right )$ is positive definite.
The column span is preserved under right multiplication by $A$, so the
positive definite condition is 
$\tr{\bigl ( \begin{smallmatrix} I \\ -CA^{-1} \end{smallmatrix}\bigr )} \cdot
\bigl ( \begin{smallmatrix} I \\ CA^{-1} \end{smallmatrix}\bigr ) \ggg 0$.
In other words $gW_\H' \in
\cB^+_L$ simply means that $gW_\H'$ is the column span of an infinite
matrix $\left ( \begin{smallmatrix} I \\ X_1  \end{smallmatrix}\right )$
such that $I - {\tr{X_1}}X_1 \ggg 0$.  Similarly $gW_\H'' \in
\cB^-_L$ simply means that $gW_\H''$ is the column span of an infinite
matrix $\left ( \begin{smallmatrix} X_2 \\ I  \end{smallmatrix}\right )$
such that $I - {\tr{X_2}}X_2 \ggg 0$.  The $G$--stabilizer
of $0 \in \cB^+_L$ is the parabolic $P$ consisting of all $\left (
\begin{smallmatrix} A & B \\ 0 & D  \end{smallmatrix}\right )$,
while the $G$--stabilizer of $0 \in \cB^-_L$ is the opposite parabolic
$\tr P = P^{opp}$ consisting of all 
$\left ( \begin{smallmatrix} A & 0 \\ C & D
\end{smallmatrix}\right )$.  
Reformulating this,

\begin{lemma} \label{bded-dom-so-quat}
Suppose that $G_0 = Sp(\infty,q)$, $q \leqq \infty $.  Then the quaternionic
positive bounded symmetric domain associated to
the triple $(V,G_0,L)$ is
$\cB^+_L \cong \{X_1 \in \H^{\infty \times q} \mid I-{\tr{X_1}}X_1 \ggg 0\}$ in
$G/P$,  and the corresponding quaternionic negative bounded symmetric 
domain for $(V,G_0,L)$ is 
$\cB^-_L \cong \{X_2 \in \H^{q\times \infty} \mid I - {\tr{X_2}}X_2 \ggg 0\}$
in  $G/P^{opp}$.
\end{lemma}
The action of $G_0$ on these bounded symmetric domains is linear fractional,
as described in Section \ref{ssec4a} for the complex case.
The proof of Theorem \ref{main-su} is valid here, giving us
the following structure theorem.

\begin{theorem}\label{main-so-quat}
Let $G_0 = Sp(\infty,q)$ with $q \leqq \infty$.  Let $D$ be an open 
$G_0$--orbit $G(\cF^{(1)})$ in the quaternionic flag manifold $\gX_{\cF,L}$.  
Then the positive definite bounded symmetric domain 
all positive definite $G$-translates of $(V_\H)_+$ and the negative definite 
bounded symmetric domain $\cB^-_L$ for $(V,G_0,L)$ is the set of is the
set of all negative definite $G$-translates of $(V_\H)_-$\,.  The
$\cB^\pm_L$ are diffeomorphic. There are three
cases for the structure of the cycle space, as follows.  

If every space $F^{(1)}_k \in \cF^{(1)}$ is positive definite then $\cM_D$
is diffeomorphic to $\cB^+_L$\,.

If every space $F^{(1)}_k \in \cF^{(1)}$ is negative definite then $\cM_D$
is diffeomorphic to $\cB^-_L$\,.

If some space $F^{(1)}_k \in \cF^{(1)}$ is indefinite then $\cM_D$ is
diffeomorphic to $\cB^+_L \times \cB^-_L$\,.
\end{theorem}

\medskip

Department of Mathematics, 

University of California,

Berkeley CA 94720--3840, USA

{\tt jawolf@math.berkeley.edu}

\end{document}